\documentclass[letterpaper,11pt]{amsart}
\usepackage[dvips]{graphicx}
\usepackage{amsmath,amsthm,amscd}
\usepackage{amssymb,amsfonts,amsxtra,dsfont}

\newcommand{\actson}{\curvearrowright}
\newcommand{\fat}[1]{\mathds{#1}}
\newcommand{\hsm}[1]{\mathcal{#1}}

\newcommand{\EE}{\fat{E}}
\newcommand{\NN}{\fat{N}}
\newcommand{\ZZ}{\fat{Z}}
\newcommand{\RR}{\fat{R}}

\renewcommand{\SS}{\fat{S}}

\newcommand{\THEN}{\;\Rightarrow\;}
\newcommand{\IFF}{\;\Leftrightarrow\;}
\newcommand{\minus}{\smallsetminus}

\newcommand{\ep}{\hfill $\square$}

\newcommand{\cl}[1]{\overline{#1}}
\newcommand{\card}[1]{\left|#1\right|}

\newcommand{\inv}{^{{\scriptscriptstyle -1}}}

\newcommand{\II}{\hsm{I}(G)}
\newcommand{\IImax}{\hsm{I}_{_{max}}(G)}
\newcommand{\bd}{\partial}
\newcommand{\bdCone}{\partial_\infty}
\newcommand{\bdTits}{\partial_{_\mathrm{T}}}
\newcommand{\clCone}[1]{\mathrm{cl}_{\infty}\!\!\left(#1\right)}
\newcommand{\clTits}[1]{\mathrm{cl_{_T}}\!\!\left(#1\right)}
\newcommand{\core}{\mathrm{core}\bdTits X}
\newcommand{\diam}{\mathrm{diam}}

\newcommand{\nbd}[1]{\mathcal{N}_{#1}}
\newcommand{\ball}[2]{B_d\!\left(#1,#2\right)}
\newcommand{\Titsball}[2]{\mathrm{B_{_T}}\!\!\left(#1,#2\right)}
\newcommand{\Tits}[2]{\mathrm{d_{_T}}\!\!\left(#1,#2\right)}
\newcommand{\minG}{\hsm{M}(G)}
\newcommand{\UF}{\beta G}
\newcommand{\ultra}[1]{T^{#1}}
\newcommand{\ulim}[2]{\displaystyle\lim_{#2\to #1}}
\newcommand{\repel}[1]{#1(-\infty)}
\newcommand{\attr}[1]{#1(\infty)}
\newcommand{\connect}[3]{#2\stackrel{#1}{\rightsquigarrow}#3}
\newcommand{\suspend}[2]{\Sigma\!\left(#1,#2\right)}
\newcommand{\equator}[2]{\mathrm{E}\!\left(#1,#2\right)}
\newcommand{\vol}[1]{\mathrm{vol}\left(#1\right)}

\newcommand{\comp}[3]{\mathrm{comp}_{#1}\left(#2,#3\right)}

\theoremstyle{plain}
\newtheorem{thm}{Theorem}[section]
\newtheorem{prop}[thm]{Proposition}
\newtheorem{lemma}[thm]{Lemma}
\newtheorem{cor}[thm]{Corollary}
\newtheorem{conjecture}[thm]{Conjecture}

\theoremstyle{definition}
\newtheorem{defn}[thm]{Definition}

\newtheorem{remark}[thm]{Remark}

\begin{document}
\title[A transversal for minsets in CAT(0) boundaries]{A `transversal' for minimal invariant sets in the boundary of a CAT(0) group}
\author[D.\ P.~Guralnik]{Dan P.\ Guralnik}
\address{Electric \& Systems Engineering\\ 
	University of Pennsylvania\\
	200 South 33rd Street,
	Philadelphia, PA 19104}
\email{guraldan@seas.upenn.edu}
\author[E.\ L.~Swenson]{Eric L.\ Swenson}
\address{Dept. of Mathematics\\
	Brigham Young University\\
	Provo, UT 84602}
\email{eric@mathematics.byu.edu}
\thanks{This work was partially supported by a grant from the Simons Foundation (209403 to Eric Swenson), and carried out while Dan Guralnik was a post-doctoral fellow at the University of Oklahoma Mathematics Department.}

\begin{abstract} We introduce new techniques for studying boundary dynamics of CAT(0) groups. For a group $G$ acting geometrically on a CAT(0) space $X$ we show there is a flat $F\subset X$ of maximal dimension (denote it by $d$), whose boundary sphere intersects every minimal $G$-invariant subset of $\bdCone X$. As applications we obtain an improved dimension-dependent bound \[\diam\bdTits X\leq 2\pi-\arccos\left(-\frac{1}{d+1}\right)\] on the Tits-diameter of $\bd X$ for non-rank-one groups, a necessary and sufficient dynamical condition for $G$ to be virtually-Abelian, and we formulate a new approach to Ballmann's rank rigidity conjectures.
\end{abstract}

\maketitle

\section{Introduction.} A central notion in the study of isometric group actions on CAT(0) spaces is the notion of rank. Recall from \cite{[Ballmann-DMV]} that a geodesic line $\ell:\RR\to X$ is said to have rank one, if $\ell$ does not bound a flat half-plane. A hyperbolic isometry of $X$ is said to have rank one, if it has a rank one translation axis. 

Given an isometric group action $G\actson X$ one says that $G$ has rank one, if some $g\in G$ acts on $X$ as a rank one hyperbolic isometry. If $G$ does not have rank one, we say that $G$ has higher rank. By a result of Ballmann (\cite{[Ballmann-Brin-orbihedra]}), the property of a geometric action of $G$ on CAT(0) space having rank one depends only on the group $G$.

Two major problems in the field are the following two conjectures, stated by Ballmann and Buyalo in \cite{[Ballmann-Buyalo-2pi]}:
\begin{conjecture}[Closing Lemma]\label{conj:closing lemma} Suppose $G$ acts properly discontinuously by isometries on a complete CAT(0) space $X$ (also known as a {\it Hadamard space}). If $\Lambda G=\bd X$ and $\diam(\bdTits X)>\pi$ then $X$ contains a $G$-periodic rank one geodesic.
\end{conjecture}
\begin{conjecture}[Rank Rigidity]\label{conj:rank rigidity} Suppose $G$ acts properly discontinuously by isometries on a complete CAT(0) space $X$ so that $\Lambda G=\bd X$. If $X$ is geodesically complete and $\diam(\bdTits X)=\pi$, then $X$ is a symmetric space or Euclidean building of higher rank, or $X$ is reducible.
\end{conjecture}

As to the closing lemma, in their paper \cite{[Ballmann-Buyalo-2pi]} Ballmann and Buyalo prove that $\bdTits X$ having diameter greater than $2\pi$ implies $G$ having rank one (in fact, their assumptions do not require the co-compactness of the action -- only that $\bd X$ coincides with the limit set of $G$). This bound on the diameter of a higher rank group was later improved by Papasoglu and Swenson in \cite{[Papasoglu-Swenson-JSJ]}, to the value $3\pi/2$. In this paper we improve this bound to \[\diam\bdTits X\leq 2\pi-\arccos\left(-\frac{1}{d+1}\right),\] where $d$ is the geometric dimension of $\bdTits X$ (theorem \ref{new diameter bound}).\\

The conjectures are resolved (in the positive) for a variety of cases, e.g. when $X$ is a Hadamard manifold and $G$ has finite co-volume (this is the original rank rigidity theorem of Ballmann -- see \cite{[Ballmann-DMV]}, theorem C), for certain $G$-co-compact cell complexes of low dimensions (see discussion in the introduction of \cite{[Ballmann-Buyalo-2pi]}), and when $X$ is a cubing and $G\actson X$ cellularly and co-compactly (Caprace-Sageev \cite{[Caprace-Sageev-rank-rigidity]}). The Caprace-Sageev approach depends very strongly on the wall structure of a cubing, and does not seem to generalize to arbitrary CAT(0) spaces.\\

The following result of Leeb indicates another possible approach:
\begin{thm}[Leeb, \cite{[Leeb]}] Let $H$ be a complete and geodesically complete locally-compact Hadamard space. If the ideal boundary $\bd H$ equipped with the Tits metric is a non-discrete join-irreducible spherical building, then $H$ is an affine building or a symmetric space.
\end{thm}
Thus, tackling the rank rigidity conjecture becomes an issue of classifying the possible Tits boundaries of proper CAT(0) spaces admitting properly-discontinuous group actions with full limit sets.\\

Taken with the Tits metric, the ideal boundary $\bd X$ of $X$ becomes the CAT(1) space $\bdTits X$, which is known to be finite-dimensional (\cite{[Kleiner-geom-dim]}, \cite{[Swenson-cut-point]}). Lytchak obtains a remarkable canonical decomposition for such spaces:
\begin{thm}[Lytchak, \cite{[Lytchak-join-decomposition]}, corollary 1.2]\label{Lytchak} Let $Z$ be a finite-dimensional geodesically complete CAT(1) space. Then $Z$ has a unique decomposition
\begin{displaymath}
	Z=\fat{S}^\ell\ast B_1\ast\ldots\ast B_k\ast Z_1\ast\ldots\ast Z_m\,, 
\end{displaymath}
where $k,m\geq 0$ and $\ell\geq -1$, $\fat{S}^\ell$ is the $\ell$-dimensional sphere of unit sectional curvature, every $B_i$ is a thick irreducible building and every $Z_j$ is an irreducible non-building. 
\end{thm}
In order to utilize this result in the context of the rank rigidity conjectures one needs to find a way to get rid of the non-building factors. In this context, the main theorem of the same paper provides a way:
\begin{thm}[Lytchak, \cite{[Lytchak-join-decomposition]}, main theorem]\label{Lytchak building characterization} Let $Z$ be a finite-dimensional geodesically complete CAT(1) space. If $Z$ has a proper closed subset $A$, containing with each $x\in A$ all antipodes of $x$ -- that is, all points $z\in Z$ with $d(x,z)\geq\pi$ -- then $Z$ is either a spherical building or a spherical join.
\end{thm}
We remark that it is not yet  known what the conditions on $X$ and/or $G$ should be in order for $Z=\bdTits X$ to be geodesically complete. However, since $X$ admits a co-compact isometric action, it is coarsely geodesically-complete (Geoghegan-Ontaneda \cite{[Geoghegan-Ontaneda-almost-geodesically-complete]}): there is a constant $R_c>0$ such that any $x,y\in X$ there is a geodesic ray emanating from $x$ and passing through the ball of radius $R_c$ about $y$.

The following result  gives a less powerful decomposition, but one that is valid without requiring geodesic completeness.  It is stated explicitly in \cite{[Swenson09]} but a curvature independent version can be easily obtained from the main result of \cite{[Foertsch-Lytchak]} using coning, and 
this  also gives a curvature independent  version of Theorem \ref{Lytchak}.
\begin{thm}[\cite{[Foertsch-Lytchak]}, \cite{[Swenson09]} theorem 14]\label{join decomposition} If $(Z,d)$ is a finite dimensional complete CAT(1) space then there is a unique decomposition of $Z$ as a metric spherical join $Z=\SS(Z)\ast E'(Z)$, where $\SS(Z)$ is a round sphere and $E'(Z)$ is a complete $\pi$-convex subspace of $Z$ which cannot be decomposed in this way. Moreover, $\SS(Z)$ coincides with the set of suspension points of $Z$. 
\end{thm}
Following \cite{[Swenson09]}, we define a suspension point of $Z$ as follows:
\begin{defn}[suspension points] Let $(Z,d)$ be a complete CAT(1) space. We say that a point $a\in Z$ is a suspension point, if there is a point $b\in Z$ with $d(a,b)=\pi$ such that $Z$ equals the union of all geodesics joining $a$ with $b$.
\end{defn}
This result plays a fundamental role in our approach to the study of $\bd X$.

\subsection{Dynamics and $G$-ultra-limits} 
We introduce a new (in the current context) technical tool for our analysis. The action $G\actson \bdCone X$ extends to an action of the Stone-\v{C}ech compactification $\beta G$ (of $G$ as a countable discrete space) on $\bdCone X$. Since $G$ acts on $\bdCone X$ discretely, the action of $\beta G$ on a point $p\in\bdCone X$ can be described as an ultra-limit of the form 
\begin{equation}
	\ultra{\omega}x=\ulim{\omega}{g}g\cdot p\,,
\end{equation}
where $\omega$ is an ultra-filter on $G$, and the orbit map $\omega\mapsto\ultra{\omega}p$ is continuous for all $p$. 

Note that in its role as an operator on $\bdCone X$ a non-principal ultra-filter $\omega$ is not necessarily continuous with respect to the cone topology. A radical example occurs when $G$ is one-ended non-elementary word-hyperbolic: since $G$ acts on $\bdCone X$ as a discrete convergence group, every non-principal $\omega\in\beta G$ has a pair of points $n,p\in\bd X$ such that $\ultra{\omega}z=p$ for all $z\neq n$; since $G$ is non-elementary, $\ultra{\omega}n\neq p$ for many choices of $\omega$.

In the more general case the situation is somewhat different. A theorem of Kleiner \cite{[Kleiner-geom-dim]} guarantees a flat in $X$, whose boundary is a round sphere of dimension equal to the geometric dimension of $\bdTits X$ (and hence not exceeding the covering dimension of $\bdCone X$). What happens to such spheres under the operators $\ultra{\omega}$?

Since the Tits metric is lower semi-continuous with respect to the cone topology, every $\omega\in\beta G$ acts on $\bdTits X$ as a $1$-Lipschitz operator. Thus, a certain degree of continuity is retained. In fact, some of these operators preserve the geometry of some such spheres:\\

\hspace{-.15in}{\bf Theorem A: }{\it 
Suppose $G$ acts geometrically on a CAT(0) space $X$ and let $d$ denote the geometric dimension of $\bdTits X$. Then for every $(d+1)$-flat $F_0\subseteq X$ there exist $\omega\in\beta G$ and a $(d+1)$-flat $F\subseteq X$ such that $\ultra{\omega}$ maps all of $\bd X$ to $\SS=\bd F$, with $\bd F_0$ mapped isometrically onto $\SS$.}\\

The immediate conclusion is that $\SS$ intersects every minimal closed $G$-invariant subset (or simply {\it minset}) of $\bdCone X$, which explains the title of this work. Theorem D (below) gives a much sharper result in this sense, producing particularly natural candidates for a transversal of the family of minsets.

Key ingredients in the proof of this result are $\pi$-convegence (see \cite{[Papasoglu-Swenson-JSJ]} or theorems \ref{original pi-convergence},\ref{pi-convergence} below), the above mentioned join decomposition theorem by Swenson, and a notion of conicity for ultra-filters (recall the notions of conical limit points {\it vs.} parabolic limit points in, say, Kleinian groups).

Theorem A demonstrates the technical power of the approach using ultra-limits: since $\bdTits X$ is rarely compact (or even locally separable), the common practice of considering sequences of elements $g_n\in G$ and passing to sub-sequences to obtain limit points of the form $\lim g_n\cdot z$ defined over Tits-compact subsets of $\bd X$ does not suffice for understanding the dynamics of the action globally. In contrast, the $1$-Lipschitz operators $\ultra{\omega}:\bd X\to\bd X$ allow one to study the geometry of $\bdTits X$ as a whole.

The use of ultra-filters and ultra-limits is not new to the field of Dynamics: it was introduced by Hindman to deal with various generalizations of Van-der-Waerden's theorem on arithmetic progressions (see survey in \cite{[Blass]}), and spawned many applications to the dynamics of commutative groups. More generally, dynamicists discussed various properties of the image of $\beta G$ in $Z^Z$ (known as Ellis' {\it enveloping semi-group} of the action $G\actson Z$) for an arbitrary (compact) dynamical system $(Z,G)$. Chapter 1 of \cite{[Glasner-Joinings]} surveys this approach to topological dynamics.

The novelty of our approach in the context of general CAT(0) groups lies in it exposing a new connection between the cone and Tits topologies on $\bd X$: while the cone topology is compact and guarantees the extension, it is the metric properties of the extension with respect to the Tits topology that prove the most useful. At the same time, $\bdTits X$ is usually not compact, which renders classical results on dynamical systems largely inapplicable to the study of the action $\beta G\actson\bdTits X$, with a lot of ground left to cover using the CAT(0) geometry of $X$ and CAT(1) geometry of $\bdTits X$, as we shall see below.

\subsection{Incompressible sets and rigidity} Recall an important notion from dynamics: two points $p,q$ in a (compact) dynamical system $(Z,G)$ are said to be {\it proximal}, if the orbit of the pair $(p,q)$ under the diagonal action $G\actson Z\times Z$ accumulates at the diagonal.

Equivalently, $p$ and $q$ are proximal if there is $\omega\in\beta G$ with $\ultra{\omega}p=\ultra{\omega}q$.

For example, any two points on the boundary of a rank one group are proximal, while no two points on the boundary of $\ZZ^n$ are.

On a boundary $\bdCone X$ of a CAT(0) group $G$ we may consider the following quantified version of proximality: for any two points $p,q\in\bd X$ and any $\omega\in\beta G$ one has
\begin{equation}
	\Tits{\ultra{\omega}p}{\ultra{\omega}q}\leq\ulim{\omega}{g}\Tits{gp}{gq}=\Tits{p}{q}\,,
\end{equation}
by lower semi-continuity of the Tits metric.
\begin{defn}[Compressibility] We say that the points $p,q\in\bd X$ form a {\it compressible pair}, if \[\Tits{\ultra{\omega}p}{\ultra{\omega}q}<\Tits{p}{q}\] holds for some choice of $\omega\in\beta G$. A subset $A\subset\bd X$ is {\it incompressible}, if no two points of $A$ form a compressible pair.
\end{defn} 
\begin{remark} A proximal pair is necessarily compressible. For a slightly less trivial example,  $\ZZ^n\actson\EE^n$ geometrically and it is easy to see that the whole ideal sphere $\bdTits\EE^n$ is incompressible.
\end{remark}
We are now able to state what is perhaps the most impressive application of theorem A in this work:\\

\hspace{-.15in}{\bf Theorem B: }{\it Suppose $G$ is a higher rank group acting geometrically on a CAT(0) space $X$. Let $\II$  denote the set of non-degenerate maximal incompressible subsets of $\bd X$. The following are equivalent:
\begin{enumerate}
	\item $G$ stabilizes (not necessarily pointwise!) an element of $\II$ ;
	\item $G$ contains a free Abelian subgroup $A$ of finite index, and $X$ contains a coarsely-dense $A$-invariant flat $F$ with compact $A$-quotient.
\end{enumerate}}

Every $\ultra{\omega}$ restricts to an isometry on an incompressible set. Theorem A then implies:\\

\hspace{-.15in}{\bf Theorem C: }{\it Suppose $G$ is a higher rank group acting geometrically on a CAT(0) space $X$ and let $d$ be the geometric dimension of $\bdTits X$. Let $\II$ denote the set of non-degenerate maximal incompressible subsets of $\bd X$. Every $A\in\II$ is a Tits-compact, connected, $\pi$-convex subset of $\bdTits X$ isometric to a compact convex spherical (possibly infinite sided) polytope.}\\

Studying the consequences of this theorem has some bearing on the rank rigidity conjectures. To start, it should be (almost) obvious that $G$ having rank one implies $\II$ is empty. One should wonder whether the converse is true.

Next, it turns out (corollary \ref{approach to closing lemma}) the closing lemma would follow if one could prove $\II$ covers $\bd X$. This motivates a deeper study of the structure of $\II$. In fact, one can focus on special elements of $\II$, producing a much sharper form of theorems A and C:\\

\hspace{-.15in}{\bf Theorem D: }{\it Suppose $G$ is a higher rank group acting geometrically on a CAT(0) space $X$ and let $d$ be the geometric dimension of $\bdTits X$. Let $d'$ denote the maximum dimension of an element of $\II$ and let $\IImax$ denote the set of all $A\in\II$ of dimension $d'$ and maximal possible ($d'$-dimensional) volume. 
Then there exist $A_0\in\IImax$ and an ultra-filter $\nu\in\beta G$ such that $\ultra{\nu}(\bd X)=A_0$.}\\

In particular, $A_0$ intersects every minset of $G$ in $\bdCone X$, and every element of $\II$ is isometric to a sub-polytope of $A_0$. As a result, all elements of $\IImax$ are isometric to each other. With theorem A in hand we may also assume $A_0$ is contained in a round sphere bounding a flat in $X$. Additional results lead us to believe that the union of $\IImax$ may, indeed, have the structure of a spherical building. Whether $\IImax$ (or even $\II$) covers $\bd X$ remains to be resolved.

At any rate, we believe our account provides sufficient evidence that the line of study proposed in this work -- perhaps applied to particular large known families of CAT(0) groups -- provides a feasible attack on the rank rigidity conjectures for CAT(0) groups, regardless of these conjectures turning out to be true or false.

\subsection{Overview of the paper} In section \ref{section:ultra-limits} we recall the construction of the extension of a discrete action $G\actson Z$ of a group on a compactum to an action of $\beta G$ on $Z$ and discuss the general properties we require later on. The exposition roughly follows that of \cite{[Blass]}, extending his discussion to the non-commutative case. This section is presented solely in the interest of completeness of the exposition, its results being already known to Dynamicists.

Section \ref{section:dynamics} discusses the general features of $\beta G\actson\bdCone X$ and how these interact with the structure of $\bdTits X$ when $G$ is a group acting geometrically on the CAT(0) space $X$.

Section \ref{section:geometry of bdTits} proves the main results and section \ref{section:discussion} discusses the relation to rank rigidity.

\subsection{Acknowledgements}
The second author wishes to thank Werner Ballmann for many helpful conversations. The first author is grateful to Boris Okun and to Uri Bader for interesting discussions and encouragement. Also, special thanks to the referee and to Aur\'elien Bosche for valuable remarks about earlier versions of this paper.

\section{Ultra-limits.}\label{section:ultra-limits} Throughout this section, $G$ is a discrete topological group acting on a compact Hausdorff space $Z$. In this section we generalize the ideas described in Blass \cite{[Blass]} in order to associate with the action $G\actson Z$ a family of ultra-limit operators $\ultra{\omega}:Z\to Z$ sending points of $Z$ to accumulation points of their orbits, as those are picked by the ultrafilter $\omega$ on $G$.

Beyond defining the $G$-ultralimits, the most important goal of this section is to investigate the extent to which the statement ``the limit of limits is itself a limit'' holds true in this context. This is where mimicking Blass' description of ultralimits over the semi-group $\NN$ turned out to be the most useful.

Recall that an {\it ultrafilter on $G$} is a finitely additive function $\omega:\mathbf{2}^G\to\mathbf{2}$ such that $\omega G=1$. An ultrafilter $\omega$ on $G$ is said to be non-principal, if $\omega F=0$ whenever $F$ is a finite subset of $G$. Denote the set of ultrafilters on $G$ by $\UF$.

Observe that a principal ultrafilter on $G$ necessarily has the form $\omega=\delta_g$ for some $g\in G$, where $\delta_g(F)=1$ iff $g\in F$. Thus, the principal ultrafilters on $G$ are in 1:1 correspondence with $G$.

In fact, $\omega:\mathbf{2}^G\to\mathbf{2}$ is an ultrafilter iff $\omega\inv(1)$ is an ultrafilter on $\mathbf{2}^G$ in the usual set-theoretic sense, with the principal ultrafilters corresponding to filters generated by the singletons of $G$. This alternative point of view implies also that every family $\hsm{F}$ of subsets of $G$ satisfying the finite intersection property (FIP) has an ultrafilter $\omega$ satisfying $\omega F=1$ for all $F\in\hsm{F}$ (the set-theoretic version of this statement is that every filter-base is contained in an ultrafilter).

For any $v=(v_g)_{g\in G}\in Z^G$ and subset $Y$ of $Z$, denote by $[v,Y]$ the set of all $g\in G$ for which $v(g)\in Y$. For the special case when $v_g=g\cdot z$ we will abuse notation, writing $[z,Y]$ to denote the set $[v,Y]$ wherever there is no possibility of ambiguity or confusion.

\begin{defn}[$\omega$-limit] For any $\omega\in\UF$, we say that a point $z_0\in Z$ is an $\omega$-limit of the vector $v\in Z^G$, if every neighbourhood $U$ of $z_0$ satisfies $\omega[v,U]=1$.
\end{defn}
Since $Z$ is Hausdorff, every $v\in Z^G$ has at most one $\omega$-limit. Now suppose some $v\in Z^G$ did not have an $\omega$-limit. Then every point $y\in Z$ has an open neighbourhood $U_y$ with $\omega[v,U_y]=0$. Equivalently, $\omega[v,F_y]=1$ where $F_y=Z\minus U_y$. Now observe that the family $\hsm{F}=\left\{F_y\right\}_{y\in Z}$ is a family of closed subsets of $Z$ having FIP. By compactness of $Z$, there is a point $z\in Z$ common to all elements of $\hsm{F}$. But then $z\in F_z$ -- contradiction. We have proved:
\begin{prop}[existence of $\omega$-limits] For every $v\in Z^G$ and $\omega\in\UF$, the $\omega$-limit of $v$ under $G$ is well-defined. We henceforth denote the $\omega$-limit of $v$ by $\ulim{\omega}{g}v_g$.
\end{prop}
\begin{remark} When $\omega$ is principal, $\omega=\delta_a$ for some $a\in G$, and then, clearly, $\ulim{\omega}{g}v_g=v_a$, as expected.
\end{remark}
A special case to consider is that of the $\omega$ limit of the orbit of a point $z\in Z$. It will be convenient to denote the action of $G$ on $Z$ by $T$ (translation), with $T^g:z\mapsto g\cdot z$ for every $g\in G$. Every $T^g$ is a homeomorphism of $Z$ onto $Z$, as it were.
\begin{defn}[$\omega$-limit along an orbit] For $z\in Z$ and $\omega\in\UF$, we define
\begin{displaymath}
	\ultra{\omega}z=\ulim{\omega}{g}T^g(z)\,.
\end{displaymath}
\end{defn}
\begin{lemma}\label{ultralimits commute with continuous maps} Suppose $Z$ and $Y$ are compact Hausdorff spaces. Then, for every $\omega\in\UF$, every $v\in Z^G$ and every continuous map $f:Z\to Y$ one has the equality
\begin{displaymath} 
	f\left(\ulim{\omega}{g}v_g\right)=\ulim{\omega}{g}f(v_g)\,.
\end{displaymath}
\end{lemma}
\begin{proof} Let $z=\ulim{\omega}{g}v_g$ and $y=f(z)$, let $U$ be a neighbourhood of $y$ in $Y$ and let $O$ be a neighbourhood of $z$ satisfying $f(O)\subset U$. By the choice of $z$ we have $\omega[v,O]=1$. Since $[v,O]\subset[f\circ v,U]$ we have $\omega[f\circ v,U]=1$ as well. Since $U$ is arbitrary, this shows $y$ is the $\omega$-limit of $f\circ v$, as desired.
\end{proof}

FROM NOW ON ASSUME $G$ IS DISCRETE.\\

For our main application this assumption holds: we will work with $Z=\hat X=X\cup\bdCone X$, and $G$ will have a non-empty regular set in $\hat X$ ($X$ is contained in the regular set of $G$, by the properness of $G\actson X$). 

The important consequence of the discreteness assumption is that the map $G\to\UF$ defined by $g\mapsto\delta_g$ is a Stone-\v{C}ech compactification of $G$, provided $\UF$ is taken with the (trace of the) Tychonoff topology.

Given $v\in Z^G$, we may consider $v$ as a continuous function from $G$ to $Z$. This means $v$ extends to a continuous function $\beta v:\UF\to Z$. 
\begin{cor}\label{formula for the ultralimit map} Suppose $v\in Z^G$. Then $v:G\to Z$ has a unique continuous extension $\beta v:\UF\to Z$, and $\beta v$ satisfies the formula
\begin{equation}
	(\beta v)(\omega)=\ulim{\omega}{g}v_g
\end{equation}
For all $\omega\in\UF$.
\end{cor}
\begin{proof}
$\beta v$ is well-defined since $\UF$ is the Stone-\v{C}ech compactification of $G$ and since $G$ is discrete. The preceding lemma applied to $\beta v$ implies that
\begin{equation}\label{eqn:formula for the ultralimit map}
	(\beta v)(\omega)=(\beta v)\left(\ulim{\omega}{g}\delta_g\right)=\ulim{\omega}{g}v_g\,
\end{equation} 
for all $\omega\in\UF$.
\end{proof}
\begin{cor}[Change of Variables]\label{cor:change of variables} Let $v\in Z^G$, $\omega\in\beta G$ and $f:G\to G$ be any function. Then
\begin{displaymath}
	\ulim{\omega}{g}v_{f(g)}=\ulim{(\beta f)\omega}{g}v_g
\end{displaymath}
\end{cor}
\begin{proof} Let $u=v\circ f$. Functoriality of $\beta$ implies $\beta u=(\beta v)\circ(\beta f)$. Then:
\begin{displaymath}
	\ulim{\omega}{g}v_{f(g)}=\ulim{\omega}{g}u_g
	=(\beta u)\omega
	=(\beta v\circ \beta f)\omega=(\beta v)\left((\beta f)\omega\right)
	=\ulim{(\beta f)\omega}{g}v_g\,,
\end{displaymath}
as desired.
\end{proof}
An additional aspect of corollary \ref{formula for the ultralimit map} is that, since $Z$ is compact Hausdorff, so is $Z^Z$ when endowed with the product topology. Thus, the representation map $T:G\to Homeo(Z)$, when viewed as a map into $Z^Z$ has a unique continuous extension $\beta T:\beta G\to Z^Z$ and the preceding corollary tells us that $\beta T(\omega)=\ultra{\omega}$ for all $\omega\in\beta G$. More generally, the continuity of the $\beta T$ is guaranteed for any topology on $Z^Z$ containing the product topology and for which $Z^Z$ remains compact. We have proved:
\begin{cor}\label{pointwise convergence of G-omegas} The family of operators $\left\{\ultra{\omega}\right\}_{\omega\in\beta G}$ is closed under pointwise limits.\qedhere
\end{cor}
We single out one more useful consequence of this last observation:
\begin{cor}\label{continuous orbit maps} Let $z\in Z$. Then the map $\omega\mapsto\ultra{\omega}z$ is a continuous map of $\beta G$ into $Z$.\qedhere
\end{cor}

\subsection{Multiplication and Antipodes.}
The space $\UF$ comes equipped with additional structure induced from the algebraic structure on $G$.

Denote the regular actions of $G$ on itself by $\ell_g(a)=ga$ and $r_g(a)=ag$. $G$ then acts on $\mathbf{2}^G$ on the right (denote the action by $\lambda:g\mapsto \lambda_g$) through precomposition with $\ell$ and on the left through precomposition with $r$ (denote by $\rho:g\mapsto \rho_g$), where the elements of $\mathbf{2}^G$ are considered as function $G\to\mathbf{2}$. 

Consequently, $G$ acts on $\UF$ on the left through precomposition with $\lambda$ (denote this action by $g\mapsto L_g$), and on the right through precomposition with $\rho$ (denoted by $g\mapsto R_g$). Overall, we have:
\begin{displaymath}
	(L_g\omega)f=1\IFF \omega(\lambda_g f)=1.
\end{displaymath}
Note that when $\omega=\delta_a$ is a principal ultrafilter:
\begin{eqnarray*}
	(L_g\delta_a)f=1&\IFF&\delta_a(\lambda_gf)=1\IFF (\lambda_gf)(a)=1\\
		&\IFF&f(ga)=1\IFF\delta_{ga}f=1\,,
\end{eqnarray*}
which implies $L_g\delta_a=\delta_{ga}$. In particular, $L_g:\UF\to\UF$ is a continuous extension of $\ell_g$ (for every $g\in G$) with respect to (the trace of) the product topology.
Thus we now have a continuous map $G\times\UF\to\UF$ defined by $g\times\omega\mapsto L_g\omega$. Now fix $\omega\in\UF$, considering it as a continuous function $R_\omega:G\to \UF$ defined by $g\mapsto L_g\omega$ (noting that $R_{\delta_a}$ extends $\rho_a$). This map extends to a continuous map of $\UF\to\UF$. The result is a map $\UF\times\UF\to\UF$, $(\sigma,\tau)\mapsto\sigma\cdot\tau$ extending the multiplication in $G$, continuous in the first coordinate (for any fixed choice of the second), and continuous in the second coordinate when the first is fixed and restricted to $G$. Let us rephrase this:
\begin{lemma}[Continuity properties of the product in $\UF$]\label{continuity of product} For every $g\in G$ and $\omega\in\UF$, the maps
\begin{displaymath}
	\left\{\begin{array}{rccl}
		L_g\,: & \UF & \longrightarrow & \UF\\
			& \sigma & \mapsto & \delta_g\cdot\sigma
	\end{array}\right.\quad\mathrm{and}\quad
	\left\{\begin{array}{rccl}
		R_\omega\,: & \UF & \longrightarrow & \UF\\
			&	\sigma & \mapsto & \sigma\cdot\omega
	\end{array}\right.
\end{displaymath}
are continuous.
\end{lemma}
We claim:
\begin{lemma}[multiplication in $\UF$] Let $a\in G$ and $\sigma,\tau\in\UF$ and $v\in Z^G$. One has the equalities:
\begin{displaymath}
	(1)\;\ulim{L_a\sigma}{g}v_g=\ulim{\sigma}{g}v_{ag}\quad
	(2)\;\ulim{\sigma\cdot\tau}{g}v_g=\ulim{\sigma}{s}\ulim{\tau}{t}v_{st}\,.
\end{displaymath}
\end{lemma}
\begin{proof} For (1) we apply the change of variables (see cor. \ref{cor:change of variables}) $f:G\to G$ given by $f=\ell_a$. Since $L_a=\beta\ell_a$, the formula follows.\\

To obtain (2), we apply (1) together with cor. \ref{formula for the ultralimit map} and lemma \ref{ultralimits commute with continuous maps}. Consider:
\begin{eqnarray*}
	\ulim{\sigma}{s}\left(\ulim{\tau}{t}v_{st}\right)
		&=&\ulim{\sigma}{s}\left(\ulim{L_s\tau}{t}v_t\right)\\
		&=&\ulim{\sigma}{s}\left((\beta v)(L_s\tau)\right)\\
		&=&(\beta v)\ulim{\sigma}{s}L_s\tau\\
		&=&(\beta v)(\sigma\cdot\tau)\\
		&=&\ulim{\sigma\cdot\tau}{g}v_g\,,
\end{eqnarray*}
which is what we wanted to prove.
\end{proof}
An immediate application to the general dynamical system $G\actson Z$:
\begin{cor}[Diagonal Principle] Suppose $Z$ is a compact Hausdorff $G$-space. Then, for every $\nu,\omega\in\UF$ and every $z\in Z$ one has $\ultra{\nu\cdot\omega}z=\ultra{\nu}\ultra{\omega}z$.
\end{cor}
\begin{proof} We have
\begin{displaymath}
	\ultra{\nu\cdot\omega}z=\ulim{\nu\cdot\omega}{p}T^p(z)
		=\ulim{\nu}{a}\ulim{\omega}{b}T^{ab}(z)
		=\ulim{\nu}{a}T^a\left(\ulim{\omega}{b}T^b(z)\right)
		=\ultra{\nu}\left(\ultra{\omega}z\right)\,.
\end{displaymath}
Note that the third equality is due to lemma \ref{ultralimits commute with continuous maps} applied to $f=T^a$. The second equality is due to the preceding lemma, of course.
\end{proof}
In other words, an action of a group $G$ on a compact Hausdorff space $Z$ by homeomorphisms extends to an action of $\beta G$ on $Z$. We remark again that $\beta G$ acts on $Z$ by maps that are not necessarily continuous.\\

Here is an example of a computation in the semi-group $\beta G$:
\begin{lemma}\label{homomorphisms extend} Let $f:G\to G$ be a homomorphism. Then $\beta f:\beta G\to\beta G$ is a homomorphism as well.
\end{lemma}
\begin{remark} The same is not true for antimorphisms -- the inversion $g\mapsto g\inv$ for example -- due to the product map $\UF\times\UF\to\UF$ having some discontinuities.
\end{remark}
\begin{proof} Let $\sigma,\tau\in\UF$, and we compute $(\beta f)(\sigma\cdot\tau)$.
\begin{displaymath}\begin{array}{rclcl}
	(\beta f)(\sigma\cdot\tau)
		&=&\ulim{\sigma\cdot\tau}{g}\delta_{f(g)}
		&=&\ulim{\sigma}{s}\ulim{\tau}{t}\delta_{f(st)}\\
		&=&\ulim{\sigma}{s}\ulim{\tau}{t}\delta_{f(s)}\cdot\delta_{f(t)}
		&=&\ulim{\sigma}{s}\ulim{\tau}{t}L_{f(s)}\delta_{f(t)}\\
		&=&\ulim{\sigma}{s}L_{f(s)}\left(\ulim{\tau}{t}\delta_{f(t)}\right)
		&=&\ulim{\sigma}{s}L_{f(s)}\left(\ulim{(\beta f)(\tau)}{t}\delta_t\right)\\
		&=&\ulim{\sigma}{s}L_{f(s)}\left((\beta f)(\tau)\right)
		&=&\ulim{\sigma}{s}\delta_{f(s)}\cdot(\beta f)(\tau)\\
		&=&R_{(\beta f)(\tau)}\left(\ulim{\sigma}{s}\delta_{f(s)}\right)
		&=&R_{(\beta f)(\tau)}\left(\ulim{(\beta f)(\sigma)}{s}\delta_s\right)\\
		&=&R_{(\beta f)(\tau)}\left((\beta f)(\sigma)\right)
		&=&(\beta f)(\sigma)\cdot(\beta f)(\tau)\,.
		\end{array}
\end{displaymath}
\end{proof}
An important anti-homomorphism for our application is the antipode: the extension of $g\mapsto g\inv$.
\begin{defn}[Antipodes] for each $\omega\in\beta G$ we define its {\it antipode} $S\omega\in\beta G$ by setting $(S\omega) F=1$ iff $\omega F\inv=1$.
\end{defn}
\begin{lemma}\label{inversion is continuous} Let $\iota$ be the unique continuous extension to $\beta G$ of the map $j:G\to\beta G$ defined by $j(g)=\delta_{g\inv}$. Then $\iota(\omega)=S\omega$ for all $\omega\in\beta G$. In particular, $\iota$ is a homeomorphism of $\beta G$ onto $\beta G$.
\end{lemma}
\begin{proof} Fix $\omega\in\beta G$ and $A\subseteq G$ and write $\omega=\ulim{\omega}{g}\delta_g$. Using lemma \ref{ultralimits commute with continuous maps} we compute:
\begin{equation*}
	\iota\omega=\iota\left(\ulim{\omega}{g}\delta_g\right)
	=\ulim{\omega}{g}\left(\iota\delta_g\right)
	=\ulim{\omega}{g}{\delta_{g\inv}}=\ulim{S\omega}{h}{\delta_h}=S\omega\,.
\end{equation*}
Finally, $\iota(\omega)=S\omega$ for all $\omega\in\beta G$ implies $\iota\circ\iota=\mathrm{id}_{\beta G}$, so $\iota$ is surjective, and is a homeomorphism. 
\end{proof}

\section{Dynamics of CAT(0) boundaries.}\label{section:dynamics} The goal of this section is to establish the basic applications of ultra-limits to boundaries of CAT(0) groups. We also revisit some of the key notions and results to reformulate them in terms of ultra-filters, for future use.

From now on, $(X,d)$ is a CAT(0) space together with a geometric action by a group $G$. Let $R>0$ be such that the ball $\ball{x}{R}$ intersects every orbit of $G$, for all $x\in X$.

By a theorem of Geoghegan and Ontaneda from \cite{[Geoghegan-Ontaneda-almost-geodesically-complete]}, $R$ may be chosen large enough so that for any $x,y\in X$ there is a geodesic ray in $X$ emanating from $x$ and passing through $\ball{y}{R}$.

\subsection{Ultra-filters, Duality, Visibility} Recall that two closed subsets of $X$ have equal limit sets in $\bdCone X$ if they are at a finite Hausdorff distance from each other. It follows that if $\omega\in\beta G$ is non-principal, then the map $x\mapsto\ultra{\omega} x$ is constant when restricted to $X$. We henceforth denote this constant by $\attr{\omega}$.
\begin{defn}[Attracting/Repelling point] Let $\omega\in\beta G$ be a non-principal ultra-filter. Then the points $\attr{\omega}$ and $\repel{\omega}=\attr{(S\omega)}$ are called the {\it attracting point} and the {\it repelling point} of $\omega$, respectively. When $p=\attr{\omega}$ and $n=\repel{\omega}$, we shall sometimes use the notation $\connect{\omega}{n}{p}$.
\end{defn}
It is a rather elementary result (see \cite{[Ballmann-DMV]}) that $G$ having rank one implies $\diam\bdTits X=\infty$. Going deeper one has the result of Ballmann and  Buyalo \cite{[Ballmann-Buyalo-2pi]}: $G$ is of higher rank if and only if $\diam\bdTits X\leq 2\pi$.  This bound was improved to $3\pi/2$ by Swenson and Papasoglu \cite{[Papasoglu-Swenson-JSJ]} using the $\pi$-convergence property. It is this property that we seek to capitalize on using ultra-filters. We quote:
\begin{thm}[\cite{[Papasoglu-Swenson-JSJ]}]\label{original pi-convergence} Let $G$ be a group acting geometrically on a CAT(0) space $X$, and let $\theta\in[0,\pi]$. Then for any sequence of distinct elements $(g_m)_{m=1}^\infty$ of $G$ there exist points $n,p\in\bd X$ and a subsequence $\left(g_{m_k}\right)_{k=1}^\infty$ such that $g_{m_k}(K)\subseteq U$ holds for every neighbourhood $U$ of $\Titsball{p}{\theta}$ and every compact $K\subseteq\bdCone X\minus\Titsball{n}{\pi-\theta}$. 
\end{thm}
Instead of reproving this result in terms of ultra-filters, it is enough to remark on how the $g_{m_k}$ and the points $p$ and $n$ are obtained. Picking a point $x$, the $g_{m_k}$ are chosen so that the sequences $\left(g_{m_k}(x)\right)$ and $\left(g_{m_k}\inv(x)\right)$ converge. The limits are denoted $p$ and $n$, respectively. It is easy to conclude that, given the sequence $(g_m\inv)$, the roles of $p$ and $n$ will be swapped. Inevitably, $\pi$-convergence has the following reformulation in terms of ultra-filters:
\begin{thm}\label{pi-convergence} Let $G$ be a group acting geometrically on a CAT(0) space $X$, let $\theta\in[0,\pi]$ and $\omega\in\UF$ non-principal. Then 
\begin{displaymath}
	\Tits{x}{\repel{\omega}}\geq\pi-\theta\THEN\Tits{\ultra{\omega}x}{\attr{\omega}}\leq\theta
\end{displaymath}
holds for every $x\in\bd X$.
\end{thm}
The $\pi$-convergence phenomenon is related to the notion of dual points introduced by Eberlein (\cite{[Eberlein-duality]}). We rewrite the original definition using ultra-filters:
\begin{defn}[Duality] A pair $(n,p)\in\bdCone X\times\bdCone X$ is said to be a {\it pair of $G$-dual points}, if there exists $\omega\in\beta G$ with $\connect{\omega}{n}{p}$.

We denote the set of dual pairs by $\hsm{D}$. Given $n\in\bdCone X$, we denote the set of all points $p\in\bdCone X$ dual to $n$ by $\hsm{D}(n)$. 
\end{defn}
\begin{lemma}[Duality properties]\label{duality properties} Suppose $G$ is a group acting discontinuously by isometries on a proper CAT(0) space $X$:
\begin{enumerate}
	\item $\hsm{D}$ is symmetric: $(n,p)\in\hsm{D}$ if and only if $(p,n)\in\hsm{D}$.
	\item \label{D} $\hsm{D}$ is a closed subspace of $\bdCone X\times\bdCone X$, invariant under the natural action of $G\times G$.
	\item $\hsm{D}(n)$ is a closed $G$-invariant subspace of $\bdCone X$ for all $n\in\bdCone X$.
\end{enumerate}
\end{lemma}
\begin{proof} Symmetry is clear, as $\connect{\omega}{n}{p}$ is equivalent to $\connect{S\omega}{p}{n}$.

If assertion \ref{D} is true, then $\hsm{D}(n)$ is the projection (to the second coordinate) of the closed set $F=\hsm{D}\cap\{n\}\times\bdCone X$. It is therefore compact, and it is $G$-invariant because $F$ is $\{1\}\times G$-invariant.

We are left to prove \ref{D}. Assume $\connect{\omega}{n}{p}$ and pick $a,b\in G$. We want to show that $b\cdot n$ and $a\cdot p$ are dual. 

Let $\nu=\delta_a\cdot\omega\cdot\delta_{b\inv}$. For any $x\in X$ one has:
\begin{displaymath}
	\ultra{\nu}x=\ultra{\delta_a}\ultra{\omega}\ultra{\delta_{b\inv}}x=
	a\cdot\left(\ultra{\omega}(b\inv\cdot x)\right)=a\cdot\attr{\omega}=a\cdot p\,,
\end{displaymath} 
and at the same time
\begin{displaymath}
	\ultra{S\nu}x=\ultra{\delta_b}\ultra{S\omega}\ultra{\delta_{a\inv}}x=
	b\cdot\left(\ultra{S\omega}(a\inv\cdot x)\right)=b\cdot\repel{\omega}=b\cdot n\,.
\end{displaymath}
Thus we have $\connect{\nu}{b\cdot n}{a\cdot p}$, as desired.

Fix $x\in X$. Recall now (corollary \ref{continuous orbit maps}) that the map $\omega\mapsto\ultra{\omega}x$ is a continuous map of $\beta G$ into $\hat X=X\cup\bdCone X$. Lemma \ref{inversion is continuous} shows that $\omega\mapsto S\omega$ is continuous as well. This implies that the map $f:\beta G\to\hat X\times\hat X$ defined by $f(\omega)=\left(\ultra{\omega}x,\ultra{S\omega}x\right)=\left(\attr{\omega},\repel{\omega}\right)$ is continuous. The set of non-principal elements of $\beta G$ is closed in $\beta G$ and therefore compact, and its image under $f$ is precisely $\hsm{D}$, which concludes the proof.
\end{proof}
Recall that the pair $(n,p)$ is said to be a {\it visible pair} (denoted $(n,p)\in\hsm{V}$) if there is a geodesic line in $X$ with endpoints $n$ and $p$. Given $n\in\bdCone X$, the set of all $p\in\bdCone X$ visible from $n$ is denoted by $\hsm{V}(n)$. 

Not much can be said about the set $\hsm{V}$ as a subspace of $\bdCone X\times\bdCone X$: the obvious traits are that $\hsm{V}$ is symmetric and invariant under the diagonal action of $G$. Using the language of ultra-filters and $\pi$-convergence we have:
\begin{lemma}\label{lemma:visible pairs} Suppose $\connect{\omega}{n}{p}$ and $q\in\hsm{V}(n)$. Then $\ultra{\omega}q=p$.\qedhere
\end{lemma}
\begin{proof} $q\in\hsm{V}(n)$ implies $\Tits{q}{n}\geq\pi$, so $\pi$-convergence (theorem \ref{pi-convergence}) implies $\Tits{\ultra{\omega}q}{p}=0$, as desired.
\end{proof}
The following is a well-known relation between visibility and duality. We re-prove it as an exercise.
\begin{lemma}[lemma 1.5 in \cite{[Ballmann-Buyalo-2pi]}] Let $\xi,\eta\in\bdCone X$ and suppose $\eta\in\hsm{V}(\xi)$. Then $\hsm{D}(\xi)\subseteq\clCone{G\cdot\eta}$.
\end{lemma}
\begin{proof} Given $\xi,\eta$ with $\eta\in\hsm{V}(\xi)$ and $\zeta\in\hsm{D}(\xi)$ we find $\nu\in\beta G$ with $\connect{\nu}{\zeta}{\xi}$. Now apply lemma \ref{lemma:visible pairs} with $\omega=S\nu$, $p=\zeta$, $n=\xi$, $m=\eta$. Then $\connect{\omega}{\xi}{\zeta}$ yields $\zeta=\ultra{\omega}\eta\in\clCone{G\cdot\eta}$. 
\end{proof}
An immediate consequence of this phenomenon has proved particularly useful, and we will be using it repeatedly:
\begin{cor}[lemma 1.6 in \cite{[Ballmann-Buyalo-2pi]}]\label{cor:minsets vs dual sets} Suppose $\xi,\eta\in\bdCone X$ satisfy $\eta\in\hsm{V}(\xi)$. If $\eta$ is contained in a minimal closed $G$-invariant set $M$ of $\bdCone X$ then $M=\hsm{D}(\xi)$.\qedhere
\end{cor}
Originally, the last lemmas were used as a tool for proving the following characterization of rank one groups in terms of minimal closed $G$-invariant subsets of $\bdCone X$. We henceforth denote by $\minG$ the set of all minimal non-empty closed $G$-invariant subsets of $\bdCone X$.
\begin{prop}[\cite{[Ballmann-Buyalo-2pi]}, proposition 1.10]\label{prop:Ballmann-Buyalo} Suppose $G$ is a group acting properly by isometries on a proper CAT(0) space $X$. If the limit set of $G$ in $\bdCone X$ equals $\bdCone X$, then the following are equivalent:
\begin{enumerate}
	\item $X$ contains a $G$-periodic rank-$1$ geodesic,
	\item $\bdTits X$ has an isolated point,
	\item Every $n\in\bdTits X$ has some $p\in\bdTits X$ with $\Tits{n}{p}>\pi$,
	\item There exist $n,p\in\bdCone X$ and $M\in\minG$ such that $\Tits{n}{p}>\pi$ and $n\in M$.
\end{enumerate} 
\end{prop}

\begin{defn}[Extreme points, Core] A point $p\in\bd X$ is said to be {\it $\epsilon$-extreme}, if $\bd X\minus\Titsball{p}{\pi+\epsilon}\neq\varnothing$. The point $p$ is {\it extreme}, if it is $\epsilon$-extreme for some $\epsilon>0$. We shall refer to the set $\core$ of non-extreme points of $\bdTits X$ as {\it the core of $\bdTits X$}.
\end{defn}
In this language, $G$ has higher rank if and only if $\core$ is non-empty, if and only if every $M\in\minG$ is contained in the core of $\bdTits X$. Also note that since $\beta G$ acts on $\bdTits X$ by $1$-Lipschitz maps, $\core$ is cone-closed, and therefore $\beta G$-invariant.

This property implies the following, more careful application of $\pi$-convergence to the notion of extremality:
\begin{lemma}\label{quantified extremality} Let $n,p\in\bd X$ and $\epsilon>0$. If $p$ is $\epsilon$-extreme and $n\in\hsm{D}{p}$ then every $z\in\bd X\minus\Titsball{n}{\pi-\epsilon}$ is extreme.
\end{lemma}
\begin{proof} We have $\Tits{p}{q}>\pi+\epsilon$ for some $q\in\bd X$. Pick $\delta>0$ such that $\Tits{p}{q}>\pi+\epsilon+\delta$. Now find $\omega\in\beta G$ with $\connect{\omega}{n}{p}$ and consider a point $z\in\bd X$ with $\Tits{z}{n}>\pi-\epsilon$. Applying $\pi$-convergence we obtain $\Tits{\ultra{\omega}z}{p}\leq\epsilon$, but then 
\begin{displaymath}
	\Tits{\ultra{\omega}z}{q}\geq\Tits{p}{q}-\Tits{\ultra{\omega}z}{p}\geq\pi+\delta>\pi\,,
\end{displaymath}
meaning $\ultra{\omega}z$ is extreme. Since $\core$ is $\beta G$ invariant, $z$ must itself be extreme, as desired.
\end{proof}
Papasoglu and Swenson prove in \cite{[Papasoglu-Swenson-JSJ]}, that if a higher rank group has no fixed point in $\bd X$ then every $M\in\minG$ satisfies the property that every point of $\bd X$ is at most $\pi/2$ away from a point of $M$. Since $M\subset\core$, they conclude that $\diam\bdTits X\leq 3p/2$. 

The requirement that $G$ not fix a point of $\bd X$ is a necessary one. Consider the example of $F_2\times\ZZ$ acting on the product of the $4$-regular tree with the real line has the spherical suspension of a Cantor set for its boundary. the suspension points are fixed, while every horizontal slice of this suspension is an element of $\minG$. However, only one of these slices is $\pi/2$-away from every point of $\bd X$.

Revisiting the argument in \cite{[Papasoglu-Swenson-JSJ]}, one sees that if any $M\in\minG$ has circumradius$<\pi/2$ then $G$ has a fixed point on $\bd X$. They then apply a result of Kim Ruane's to conclude that $G$ is virtually a product of the form $H\times\ZZ$. This results in $\bdTits X$ having diameter exactly $\pi$ (which is less than $3\pi/2$).

Keeping this picture in mind, we would like at this point to present an application of theorem A to the problem of (slightly) improving the $3\pi/2$ bound on the diameter.
\begin{thm}\label{new diameter bound} Suppose $G$ is a higher rank CAT(0) group acting geometrically on a CAT(0) space $X$, and let $d\geq 1$ denote the geometric dimension of $\bdTits X$. Then $\diam\bdTits X\leq 2\pi-\arccos\left(-\frac{1}{d+1}\right)$.
\end{thm}
\begin{proof} Fix a $(d+1)$-flat $F_0$ and let $F,\SS$ and $\omega$ be as provided by theorem A.

Suppose $\bdTits X$ contains an $\epsilon$-extreme point $p$ for some $\epsilon>0$. Since $\hsm{D}(p)$ is $G$-invariant and cone-closed, $\hsm{D}(p)$ contains some $M\in\minG$. 

We consider the set $N=M\cap\SS$. Since $\ultra{\omega}M\subseteq M$ and $\ultra{\omega}M\subset\SS$, the set $N$ is non-empty and closed. Moreover, applying lemma \ref{quantified extremality} we conclude that $\diam N\leq\pi-\epsilon$, since $N\subset\core$.

For each $n\in N$ let $H_n$ denote the closed hemisphere of $\SS$ centered at the point $n$. For any point $z\in\SS$ find a point $z'\in\bd F_0$ satisfying $\ultra{\omega}z'=z$ and a point $m\in M$ with $\Tits{m}{z'}\leq\pi/2$. Then the point $n=\ultra{\omega}m$ lies in $N$ and we conclude that $z\in H_n$. Thus, the collection $\{H_n\}_{n\in N}$ covers $\SS$.

By classical geometry of round spheres, if $\pi-\epsilon<\arccos\left(-\frac{1}{d+1}\right)$ then the circumradius of $N$ in $\SS$ is less than $\pi/2$. In this case $N$ will have a circumcenter $z$ in $\SS$, implying that the antipode of $z$ on $\SS$ is not covered by any of the hemispheres $H_n$ -- a contradiction. Thus, no point of $\bd X$ is $\epsilon$-extreme for $\epsilon=\pi-\arccos\left(-\frac{1}{d+1}\right)$. Equivalently, the diameter of $\bdTits X$ cannot exceed $2\pi-\arccos\left(-\frac{1}{d+1}\right)$.
\end{proof}
\begin{remark} The function $2\pi-\arccos\left(-\frac{1}{d+1}\right)$ is an increasing function of $d$, with a value of $4\pi/3$ for $d=1$ and a limit of $3\pi/2$ as $d\to\infty$. For now, in order to obtain a better bound than $3\pi/2$ on $\diam\bdTits X$ one needs to restrict the geometric dimension of $X$.
\end{remark}

\subsection{Pulling and Suspensions} We are now after a sharpened form of lemma \ref{lemma:visible pairs}. Given a geodesic ray $\gamma:[0,\infty)\to X$ and a point $x\in X$, let $\hsm{A}_{x,\gamma,C}$ denote the family of all sets of the form
\begin{displaymath}
	A_{x,\gamma,M,C}=\left\{g\in G\left|g\cdot x\in T_{\gamma,M,C}\right.\right\}\,,
\end{displaymath}
where $T_{\gamma,M,C}=\nbd{C}\left(\gamma\left([M,\infty)\right)\right)$, with $M\in(0,\infty)$.

\includegraphics[keepaspectratio, scale =1.7]{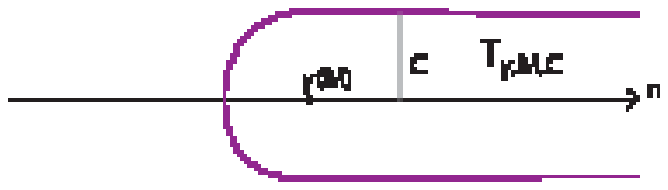}

\begin{defn}[Pulling]\label{defn:pulling} We say that $\omega\in\beta G$ pulls away from a point $n\in\bdCone X$ if there exist a point $x\in X$, a ray $\gamma$ with $\gamma(\infty)=n$ and $C>0$ such that $(S\omega) A=1$ for all $A\in\hsm{A}_{x,\gamma,C}$.
\end{defn}

\begin{remark}\label{pulling is independent of base point} The definition of pulling is independent of the choice of the point $x$: at the price of replacing $C$ by $C+d(x,y)$ one can always replace the point $x$ in the definition by any point $y\in X$.
\end{remark}
\begin{remark} It is easy to see that $\hsm{A}_{x,\gamma,C}$ forms a chain of non-empty sets whenever $C>R$ (where $R$ is as defined in the beginning of the section). In particular $\hsm{A}_{x,\gamma,C}$ has FIP (given $C>R$).


This means that ultra-filters pulling from $n$ do exist for any $n\in\bdCone X$, provided $G$ acts geometrically on $X$. 

Also note that any cone neighbourhood of $n=\gamma(\infty)$ contains $T_{\gamma,M,C}$ for $M$ large enough, so that $\attr{S\omega}=n$ is guaranteed to hold whenever $\omega$ pulls from $n$. 

Thus, one should think of an ultra-filter $\omega$ pulling from a point $n\in\bdCone X$ as a means of discussing radial convergence of $G\cdot x$ to $n$. If the action of $G$ on $X$ is only proper, then it would be sensible to define {\it radial limit points of $G$} as points $n\in\bdCone X$ admitting an $\omega\in\beta G$ pulling from $n$.
\end{remark}

\begin{lemma}[Pulling lines]\label{pulling lines} Suppose $\omega\in\beta G$ pulls from a point $n\in\bdCone X$ and $\connect{\omega}{n}{p}$. Then $\ultra{\omega}n\in\hsm{V}(p)$ and $\hsm{V}(n)\neq\varnothing$.
\end{lemma}
\begin{remark} If $G$ acts geometrically on $X$, every point of $\bd X$ has some $\omega\in\beta G$ pulling away from it. We conclude that every point of $\bd X$ is visible from some other point of $\bd X$.
\end{remark}
\begin{proof} First we note that $p=\ultra{\omega}x=\ultra{\omega}\gamma(0)$. Fix $C\geq R$ as in the definition of pulling. For every $s\in\RR$ we define an operator $P_s:X^{[0,\infty)}\to X^{\RR}$ by setting $(P_sf)(t)=f(0)$ for $t\leq -s$ and $P_sf(t)=f(t+s)$ for $t\geq -s$.

For every $g\in G$, let $s(g)$ denote the real number $s$ such that $\gamma(s)$ is the closest point projection of $g\inv\cdot x$ onto $\gamma$. 

\includegraphics[keepaspectratio, scale =1.8]{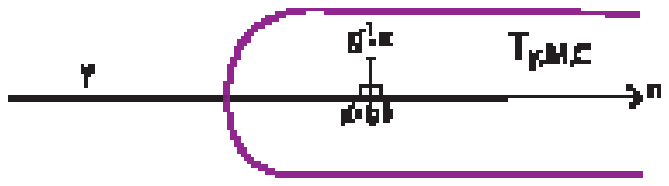}

Define a function $\ell_g:\RR\to X$ by setting $\ell_g=P_{s(g)}(g\cdot\gamma)$, and set $\ell(t)=\ulim{\omega}{g}\ell_g(t)$. We claim that $\ell$ is a geodesic line in $X$ with $\ell(-\infty)=p$ and $\ell(\infty)=\ultra{\omega}n$.

First observe that $\ell(t)\in X$ for all $t\in\RR$: for each $g\in G$, $\ell_g(0)$ belongs to the closed ball of radius $r_g=d(g\inv\cdot x,\gamma)$ about $x$, and we have $r_g\leq C$ for a set of elements in $G$ having full $\omega$-measure. As a result, $\ell_g(t)$ belongs to the closed ball of radius $C+t$ about $x$ for a set of $g\in G$ of full $\omega$-measure, which proves $\ell(t)$ must belong to this same ball.

To prove that $\ell$ is a geodesic, it suffices to show that for any reals $t_1<t_2$, the set of $g\in G$ satisfying $d\left(P_{s(g)}(g\cdot\gamma)(t_1),P_{s(g)}(g\cdot\gamma)(t_2)\right)=t_2-t_1$ has full $\omega$-measure. For this it suffices to show that the set of $g\in G$ for which $-s(g)\leq\min\{t_1,t_2\}$ has full $\omega$-measure, but this is guaranteed by $\omega$ pulling from $n$.

We are left to prove that $\ell(\infty)=\ultra{\omega}n$ and $\ell(-\infty)=p$.

Let $A$ be the set of all $g\in G$ such that $d(\ell(0),g\cdot\gamma(s(g)))<2C$ and let $B_M$ be the set of all $g\in G$ such that $d(\ell(M),g\cdot\gamma(s(g)+M))<2C$. By the definition of an $\omega$-limit, we have $\omega A=1$ and $\omega B_M=1$, implying $\omega(A\cap B_M)=1$. Thus, for every $M>0$ the ray $g\cdot\gamma$ fellow-travels the ray $\ell\big|_{[0,\infty)}$ along $\ell\big|_{[0,M]}$ for every $g$ belonging to a set of full $\omega$-measure. This implies $\ultra{\omega}n=\ell(\infty)$.

At the same time, given $M>0$ we may consider the set $B'_M$ of all $g\in G$ such that $d(\ell(-M),g\cdot\gamma(s(g)-M))<2C$. Similarly to the above, $A\cap B'_M$ has full $\omega$-measure, proving that the geodesic interval $[x,g\cdot\gamma(0)]$ fellow-travels the geodesic ray $\ell(-t)$, $t\geq 0$ for $t\in[0,M]$. Since $\ultra{\omega}\gamma(0)=p$, this proves $p=\ell(-\infty)$.

The second claim follows from $X$ being almost geodesically-complete using a similar technique. We pick for every $g\in G$ a geodesic ray $\mu_g:[0,\infty)\to X$ emanating from $x$ and passing through $\ball{g\inv\cdot\gamma(0)}{R}$. Observe that $\ulim{S\omega}{g}\mu_g(0)=n$.

This time let $s'(g)$ be the smallest value of $s\in[0,\infty)$ for which $\mu_g(s)$ is the closest point projection of $g\inv\cdot\gamma(0)$ to $\mu_g$. Equivalently, $s'(g)$ is the smallest value of $s$ for which $g\cdot\mu_g(s)$ is the closest point projection of $\gamma(0)$ to the ray $g\cdot\mu_g$. Thus $s'(g)$ is an unbounded function of $g$ on a set of full $S\omega$-measure, by the previous observation.

Once again, we define $\ell'_g=P_{s'(g)}(g\cdot\mu_g)$ and let $\ell'(t)=\ulim{S\omega}{g}\ell'_g(t)$. Using the same argument as in the first part of the proof, we conclude $\ell'$ is a geodesic line with $\ell'(0)$ at most $R$ away from $\gamma(0)$ and with $\ell'(-\infty)=n$. In particular we have $\ell'(\infty)\in\hsm{V}(n)$, as desired.
\end{proof}

\begin{lemma}[Pulling Flats]\label{pulling flats} Suppose $S\subset\bdTits X$ is a sphere bounding a flat $F\subset X$, and let $\ell\subset F$ be a line joining a pair of points $m,n\in S$. If $\omega\in\beta G$ pulls away from $n$, then $\ultra{\omega}S$ is a sphere bounding a flat isometric to $F$. In particular, one has:
\begin{displaymath}
	(1)\;\repel{\omega}=n\,,\quad
	(2)\;\attr{\omega}=\ultra{\omega}m\,,\;\mathrm{and}\quad
	(3)\;\Tits{\ultra{\omega}\repel{\omega}}{\attr{\omega}}=\pi\,.
\end{displaymath}
\end{lemma}
\begin{proof} For the equalities above, observe that (1) follows from pulling, (2) follows from $\pi$-convergence as $\Tits{n}{m}=\pi$, and (3) follows from the main assertion of the lemma.

The rest of the proof proceeds in the same manner as that of the pulling property for lines. Fix an isometry $f$ of $\EE^d$ onto $F$, $d\geq 1$, and write points of $\EE^d$ in the form $(v,t)$ with $v\in\EE^{d-1}$ and $t\in\EE^1$, so that $f(0,t)$ parametrizes the line $\ell\subseteq F$, with $f(\infty)=n$.

For any $g:\EE^d\to X$, any $s\in\RR$ and $(u,t)\in\EE^d$ define $(P_sg)(u,t)=g(u,s+t)$. For any $g\in G$, let $s(g)$ denote the unique value of $s\in\RR$ for which $f(0,s)$ equals the projection of the point $g\inv(s)$ to the line $\ell$.

Similarly to the previous lemma we define $f_g=P_{s(g)}(g\cdot f)$ and observe that, for every $(u,t)\in\EE^d$, the function $(f_g(u,t))_{g\in G}$ from $G$ to $X$ is bounded. This implies that the pointwise limit $f_\omega=\ulim{\omega}{g}f_g$ is a well-defined map of $\EE^d$ into $X$.

The limit of isometries is an isometry: for any pair of points $p_i=(v_i,t_i)\in\EE^d$ ($i=1,2$) we have $d\left(f_g(p_1),f_g(p_2)\right)=d(p_1,p_2)$, so that $\ulim{\omega}{g}d\left(f_g(p_1),f_g(p_2)\right)=d(p_1,p_2)$ and $f_\omega$ is an isometry of $\EE^d$ into $X$.
\end{proof}
Boundaries of flats are a particular case of the following notion:
\begin{defn} Let $n\in\NN\cup\{0\}$. A round $n$-sphere in $\bdTits X$ is an isometrically embedded copy of the standard unit curvature sphere $\SS^n\subset\EE^{n+1}$ in $\bdTits X$.
\end{defn}
We require some notation for suspensions:
\begin{defn} Let $(Z,\rho)$ be a complete CAT(1) space and let $p,q\in Z$ satisfy $\rho(p,q)=\pi$. Define the following subsets:
\begin{eqnarray}\label{eqn:suspension and equator}
	\suspend{p}{q}&=&\left\{x\in Z\,\big|\,\rho(p,x)+\rho(x,q)=\pi\right\}\,,\\
	\equator{p}{q}&=&\left\{x\in Z\,\big|\,\rho(p,x)=\rho(x,q)=\pi/2\right\}\,.
\end{eqnarray}
The points $p$ and $q$ are referred to as the {\it poles of $\suspend{p}{q}$}. A geodesic segment with endpoints $p$ and $q$ will be referred to as a {\it longitude} of $\suspend{p}{q}$. For a subset $S\subseteq Z$ we shall say that $p,q\in Z$ is a pair of suspension points for $S$ if $S\subset\suspend{p}{q}$.
\end{defn}
The notions of poles and longitudes of $\suspend{p}{q}$ are obviously ambiguous unless the points $p$ and $q$ are specified. The following lemma explains the structure of suspensions a little bit better:
\begin{lemma}\label{lemma:suspension} Let $(Z,\rho)$ be a complete CAT(1) space and suppose $p,q\in Z$ satisfy $\Tits{p}{q}=\pi$. Then:
\begin{enumerate}
	\item Both $\equator{p}{q}$ and $\suspend{p}{q}$ are closed and $\pi$-convex;
	\item $\suspend{p}{q}$ is naturally isometric to the metric spherical suspension $\{p,q\}\ast\equator{p}{q}$.
\end{enumerate}
\end{lemma}
\begin{proof} Closedness of $\equator{p}{q}$ and $\suspend{p}{q}$ follows immediately from the definitions. Suppose $a,b\in\equator{p}{q}$ satisfy $\rho(a,b)<\pi$. Then the triangle $\vartriangle pab$ has perimeter strictly less than $2\pi$, implying that a comparison triangle for $\vartriangle pab$ is contained in a hemisphere $H$ of $\SS^2$, with $p$ at the pole of $H$. This then gives $\rho(p,z)\leq\pi/2$ for all $z\in[a,b]$. Similarly, for $\vartriangle qab$ we obtain $\rho(q,z)\leq\pi/2$. Together with $\rho(p,q)=\pi$ and the triangle inequality, we obtain $z\in\equator{p}{q}$ for all $z\in[a,b]$, as required.

By lemma 4.1 of \cite{[Lytchak-join-decomposition]}, we conclude that $\suspend{p}{q}$ is isometric to the metric spherical join $\{p,q\}\ast\equator{p}{q}$. Therefore, by uniqueness of geodesics of length less than $\pi$, $\suspend{p}{q}$ must be $\pi$-convex.
\end{proof}

\begin{lemma}\label{embedding lemma} Suppose $\omega\in\beta G$ satisfies $\connect{\omega}{n}{p}$ with $\pi\leq\Tits{\ultra{\omega}n}{p}<\infty$. Then $\Tits{\ultra{\omega}n}{p}=\pi$, and for every $a\in\Titsball{n}{\pi}$, the map $f:x\mapsto\ultra{\omega} x$ restricts to an isometry on the geodesic segment $[n,a]$.
\end{lemma}
\begin{proof} It suffices to prove $\Tits{f(n)}{f(a)}=\Tits{n}{a}$. The rest will follow from $f$ being $1$-Lipschitz. Let $\theta=\Tits{n}{a}$ then, by $\pi$-convergence we have $\Tits{p}{f(a)}\leq\pi-\theta$. We also have $\Tits{f(n)}{f(a)}\leq\Tits{n}{a}=\theta$. This gives:
\begin{displaymath}
	\pi=\theta+(\pi-\theta)\geq\Tits{f(n)}{f(a)}+\Tits{f(a)}{p}\geq\Tits{f(n)}{p}\geq\pi\,.
\end{displaymath}
Thus, we have equalities throughout, making $\Tits{f(n)}{f(a)}<\theta$ impossible.
\end{proof}
\begin{cor}\label{UFs map to suspensions} Suppose $\connect{\omega}{n}{p}$ satisfies $\pi\leq\Tits{\ultra{\omega}n}{p}<\infty$. Then $\ultra{\omega}$ maps $\bdTits X$ into the suspension $\suspend{\ultra{\omega}n}{p}$. In particular this holds whenever $\omega$ pulls from $n$ and there exists a point $m\in\bd X$ with $\Tits{n}{m}=\pi$.
\end{cor}
\begin{proof} Simply observe that every point of $\bdTits X\minus\Titsball{n}{\pi}$ is mapped to $p$, by $\pi$-convergence, while the preceding lemma shows that the Tits ball of radius $\pi$ about $n$ is mapped into $\suspend{\ultra{\omega}n}{p}$.

Under the additional assumptions we see that $\pi\leq\Tits{\ultra{\omega}n}{p} $ is guaranteed by lemma \ref{pulling lines}, while the requirement $\Tits{\ultra{\omega}n}{p}<\infty$ follows from $\pi$-convergence: since $\Tits{n}{m}=\pi$, we must have $\ultra{\omega}q=p$, so that \[\Tits{\ultra{\omega}n}{p} = \Tits{\ultra{\omega}n}{\ultra{\omega}q} \le\Tits{n}{m}=\pi.\]
\end{proof}
The main application of the last corollary is the following theorem -- the sharper form of theorem A from the introduction:
\begin{thm}[Folding Lemma]\label{basic tool} Suppose $G$ acts geometrically on a CAT(0) space $X$ and let $d$ denote the geometric dimension of $\bdTits X$. Then for every $(d+1)$-flat $F_0\subseteq X$ there exist $\omega\in\beta G$ and a $(d+1)$-flat $F\subseteq X$ such that $\ultra{\omega}$ maps all of $\bd X$ to $\SS=\bd F$, with $\bd F_0$ mapped isometrically onto $\SS$.
\end{thm}
\begin{remark} Flats of dimension $(d+1)$ are guaranteed by a result of Kleiner from \cite{[Kleiner-geom-dim]}. This produces theorem A.
\end{remark}
\begin{proof}{\it of \ref{basic tool}} Fix the $(d+1)$-dimensional flat $F_0$. If $d=0$, then $G$ is word-hyperbolic (since $X$ contains no $2$-flat) and any pair of points could serve as $\SS$. Thus we may assume $d>0$.

Let $A^0=\left\{z^0_1,\ldots,z^0_{d+1}\right\}$ be a subset of $\bd F^0$ such that $\Tits{z^0_i}{z^0_j}=\pi/2$ whenever $i\neq j$.

We construct $d$-dimensional round spheres $\SS^1,\ldots,\SS^{d+1}\subset\bd X$ inductively. Note that superscripts indicate indexation -- not dimension.

Given $k\geq 0$, a $(d+1)$-flat $F^k$ with $\bd F^k=\SS^k$ and a set $A^k$ of points $\left\{z^k_i\right\}_{i=1,\ldots,d+1}$ satisfying $\Tits{z^k_i}{z^k_j}=\pi/2$ whenever $i\neq j$, we find $\omega^k\in\beta G$ with $\connect{\omega^k}{z^k_{k+1}}{p^{k+1}}$, pulling from $z^k_{k+1}$, with $\ultra{\omega^k}$ mapping $X$ into the metric spherical suspension $\Sigma^{k+1}=\suspend{z^{k+1}_{k+1}}{p^{k+1}}$ (cor. \ref{UFs map to suspensions}).

By lemma \ref{pulling flats}, the set $\SS^{k+1}=\ultra{\omega^k}\SS^k$ is the boundary of a $(d+1)$-flat $F^{k+1}$ -- hence a round sphere. Moreover, since $\ultra{\omega^k}$ maps $\SS^k$ isometrically onto $\SS^{k+1}$ (a surjective Lipschitz map between isometric compact spaces is an isometry), the points $z^{k+1}_i=\ultra{\omega^k}z^k_i$ lie at distances of $\pi/2$ from each other as well, and we denote $A^{k+1}=\ultra{\omega^k}A^k$.

For every $k$, we denote the unique antipode of $z^k_i$ on the sphere $\SS^k$ by $-z^k_i$. If $B\subset A^k$ then let $-B$ denote the set of all $-z^k_i$ such that $z^k_i\in B$, and set $\pm B$ to equal $B\cup-B$. Note that $p^k=-z^k_k$ for all $k$.\\

Let $M^0=\bd X$, and for each $k\leq 1$ let $M^{k+1}=\ultra{\omega^k}M^k$. We have $\SS^k\subset M^k$ for all $k$, and $M^k\subset\Sigma^k$ for $k\geq 1$. Our goal is to show that $M^{d+1}\subset\SS^{d+1}$, because then choosing $\omega=\omega^d\cdot\omega^{d-1}\cdots\omega^0$ produces the required result.\\

Let $B^k$ denote the subset of $A^k$ consisting of the points $z^k_1,\ldots,z^k_k$. We claim that the points of $\pm B^k$ are suspension points of $M^k$ for every $k\geq 1$. For $k=1$ we already have $M^1\subseteq\suspend{z^1_1}{p^1}$, and $p^1=-z^1_1$. Suppose our claim is true for some $1\leq k\leq d$. We must show that $M^{k+1}\subset\suspend{z^{k+1}_i}{-z^{k+1}_i}$ holds for all $i\leq k+1$. This is surely true for $i=k+1$ since we already have $M^{k+1}\subset\Sigma_{k+1}$. Now $B^{k+1}=\pm\ultra{\omega^k}B^k\cup\{\pm z^{k+1}_{k+1}\}$, and we have $M^k\subseteq\suspend{z}{-z}$ for every $z\in B^k$, by the induction hypothesis. Then $M^{k+1}=\ultra{\omega^k}M^k\subseteq\ultra{\omega^k}\suspend{z}{-z}$ for all $z\in B^k$. However, by construction, $\ultra{\omega^k}z$ and $\ultra{\omega^k}(-z)$ are antipodes in $\SS^{k+1}$, which implies that $\ultra{\omega^k}\suspend{z}{-z}$ maps into $\suspend{\ultra{\omega^k}z}{\ultra{\omega^k}(-z)}$, and the induction argument is complete.\\

Our conclusion from the preceding paragraph is that $\pm B^{d+1}=\pm A^{d+1}$ is contained in the set of suspension points of $M^{d+1}$, implying $\SS^{d+1}$ is contained in the set of suspension points of $M^{d+1}$. Let $Y$ denote the intersection of all suspensions $\suspend{p}{q}$ containing $M^{d+1}$. Then $Y$ is a closed, $\pi$-convex subspace of $\bdTits X$ and $\SS^{d+1}$ consists of suspension points of $Y$. Using theorem \ref{join decomposition} again, write $Y=\SS(Y)\ast E'(Y)$. Since $\SS^{d+1}\subset M^{d+1}\subset Y$ and $\SS^{d+1}\subset\SS(Y)$, the fact that $\SS^{d+1}$ is a sub-sphere of dimension equal to the geometric dimension of $\bdTits X$ implies $E'(Y)$ must be empty, proving $M^{d+1}=\SS^{d+1}$, as required.
\end{proof}
The spheres $\SS$ supplied by the theorem are not transverse to the family of non-empty minimal closed $G$-invariant subsets of $G$ in $\bdCone X$, but seem to be quite close to achieving this goal:
\begin{cor} Suppose $G$ is a group acting geometrically on a CAT(0) space $X$. Let $\SS\subset\bdTits X$ be a round sphere as in the conclusion of theorem \ref{basic tool}. Then $\SS$ intersects every minimal closed $G$-invariant subset $M$ of $\bdCone X$.

Moreover, if $G$ does not fix a point of $\bd X$, then every point of $\SS$ is at a distance at most $\pi/2$ away from a point of $M\cap\SS$. In particular, $M$ intersects $\SS$ in at least two points.
\end{cor}
\begin{proof} Let $M$ be a non-empty minimal closed $G$-invariant subset of $\bdCone X$. Let $F_0,F,\SS$ and $\omega\in\beta G$ be as in theorem \ref{basic tool}. Since $M\cap\SS$ contains $\ultra{\omega}(M)$, we have that $M$ intersects $\SS$. 

If $p\in M\cap\SS$, let $q$ be the antipode of $p$ on $\SS$ and let $q_0$ be the only point of $\bd F_0$ satisfying $\ultra{\omega}q_0=q$. By theorem 23 of \cite{[Papasoglu-Swenson-JSJ]}, $\Tits{n}{M}\leq\pi/2$ for all $n\in\bd X$. In particular, there is a point $m\in M$ satisfying $\Tits{q_0}{m}\leq\pi/2$. But then we have $\ultra{\omega}{m}\in M\cap\SS$ together with $\Tits{q}{\ultra{\omega}m}\leq\pi/2$, implying $\ultra{\omega}{m}\neq p$. Now we have two points of $M$ in $\SS$, as desired.
\end{proof}

\section{Geometry of the Tits Boundary}\label{section:geometry of bdTits}
We begin by studying the basic connections between higher rank and incompressible subsets of $\bdTits X$. From this point on let $\II$ denote the set of all non-degenerate maximal incompressible subsets of $\bdTits X$, and let $\minG$ denote the set of all minimal non-empty closed $G$-invariant subsets of $\bdCone X$. We will also assume $\bd X$ contains at least three distinct points. 

\subsection{Compressibility, Collapsibility and Rank}\label{subsec:compressibility and collapsibility} We recall the definitions from the introduction and add some new ones:

\begin{defn} A pair of points $p,q\in\bdCone X$ is {\it proximal}, if there exists $\omega\in\beta G$ satisfying $\ultra{\omega}p=\ultra{\omega}q$. The pair $p,q$ is {\it compressible}, if $\Tits{\ultra{\omega}p}{\ultra{\omega}q}<\Tits{p}{q}$ for some $\omega\in\beta G$. 

A set $A$ is {\it compressible (collapsible)} if it contains a compressible (resp. collapsible) pair. We say that $A$ is {\it strongly collapsible} if there is an $\omega\in\beta G$ such that $\ultra{\omega}(A)$ is a singleton. 
\end{defn}
Note that every incompressible subset of $\bdTits X$ is contained in a maximal one, by Zorn's lemma.

The most extreme example of collapsibility is seen in the rank one case:
\begin{lemma}\label{rank one is strongly collapsible} If $G$ has rank one (and $\card{\bd X}\geq 3$) then $\bd X$ is strongly collapsible.
\end{lemma}
\begin{proof} First of all we remark that $G$ cannot possibly contain a finite-index subgroup fixing a point of $\bd X$ (otherwise apply the result of Ruane telling us that $G$ virtually splits as a product of the form $H\times\ZZ$ and hence has higher rank).

Let $g\in G$ be a rank one element and let $p$ and $n$ denote the attracting and repelling fixed points of $g$, respectively. Pick an $\omega\in\beta G$ with $\omega\left\{g^n\,\big|\,n\geq k\right\}=1$ for all $k\in\ZZ$. Then $\connect{\omega}{n}{p}$ holds, implying $\ultra{\omega}q=p$ whenever $q\in\bd X\minus\{n\}$. In addition, $\ultra{\omega}$ fixes both $n$ and $p$.

{\bf First suppose some $M\in\minG$ avoids $n$. } Then $p=\ultra{\omega}M\in M$. Since $M=\{p\}$ is impossible, there is a point $z\in M\minus\{n,p\}$. Since $z$ is visible from $n$ ($n$ is isolated in $\bdTits X$), we must have $\hsm{D}(n)\subseteq\cl{G\cdot z}\subseteq M$. Thus, $z\in\hsm{D}(n)$, by the minimality property of $M$.

Let $U,V,W$ be closed, pairwise disjoint cone neighbourhoods of $n,p,z$ respectively. Since $z$ and $n$ are joined by a rank one geodesic, there exists a rank one element $h\in G$ with $h^\infty\in U$ and $h^{-\infty}\in W$. The same as earlier in the case of $g$, there exists $\nu\in\beta G$ fixing both $h^{\pm\infty}$ and satisfying $\connect{\nu}{h^{-\infty}}{h^\infty}$ and hence $\ultra{\nu}q=h^\infty$ whenever $q\neq h^{-\infty}$. In particular, $\ultra{\nu}n=\ultra{\nu}p=h^\infty$, and we conclude that $\ultra{\nu\cdot\omega}\bd X=\{h^\infty\}$, as desired.

{\bf If some $M\in\minG$ avoids $p$, } then a symmetric argument to the above provides us with the desired collapse.

{\bf Assume every $M\in\minG$ contains both $n$ and $p$. } Then there is only one $M\in\minG$. However, $M\neq\{n,p\}$ so there is a point $z\in M\minus\{n,p\}$. Since $\hsm{D}(z)$ is closed and $G$-invariant, $\hsm{D}(z)$ contains $M$. However, with $n$ visible from $z$ we also conclude $\hsm{D}(z)\subseteq\cl{G\cdot n}=M$. In particular, $z$ is self-dual and there is $\nu\in\beta G$ with $\connect{\nu}{z}{z}$. Finally, since both $n$ and $p$ are visible from $z$, we have $\ultra{\nu}n=\ultra{\nu}p=z$, and hence $\ultra{\nu\cdot\omega}\bd X=\{z\}$. We are done.
\end{proof}
\begin{center}
HENCEFORTH, $G$ HAS HIGHER RANK UNLESS STATED OTHERWISE
\end{center}
An incompressible set is obviously non-collapsible. We take some time to study the interplay between compressibility and collapsibility. For that we need --
\begin{defn}[Compression] Let $x,y\in\bd X$ and $\omega\in\UF$. Define the $\omega$-compression of $x,y$ to be:
\begin{displaymath}
	\comp{\omega}{x}{y}=\Tits{x}{y}-\Tits{\ultra{\omega}x}{\ultra{\omega}y}.
\end{displaymath}
The $G$-compression of $x,y$ will be defined as
\begin{displaymath}
	\comp{G}{x}{y}=\max_{\omega\in\beta G}\comp{\omega}{x}{y}.
\end{displaymath}
If $\comp{\omega}{x}{y}=\comp{G}{x}{y}$ we will say that $\omega$ achieves maximal compression of the pair $x,y$.
\end{defn}
Observe that the $G$-compression of $x$ and $y$ is well-defined (always achieved): indeed, the function $\omega\mapsto\comp{\omega}{x}{y}$ is upper semi-continuous, so it must achieve its supremum on the compact space $\beta G$. Another observation is that
\begin{displaymath}
	\comp{\nu\cdot\omega}{x}{y}=
	\comp{\omega}{x}{y}+\comp{\nu}{\ultra{\omega}x}{\ultra{\omega}y}\,,
\end{displaymath}
which immediately implies:
\begin{lemma} Either there exists a non-degenerate incompressible set in $\bdTits X$, or any pair of points in $\bd X$ is collapsible.
\end{lemma}
\begin{proof} Suppose that no pair of distinct points in $X$ is incompressible. Then for every pair of distinct points $x,y\in\bd X$ we have $\comp{G}{x}{y}>0$. Let $\omega\in\beta G$ achieve maximal compression for $x,y$. Then for $\nu\in\UF$ achieving maximal compression of the pair $\ultra{\omega}x,\ultra{\omega}y$ we have: 
\begin{eqnarray*}
\comp{\nu\cdot\omega}{x}{y}
	&=&\comp{\omega}{x}{y}+\comp{\nu}{\ultra{\omega}x}{\ultra{\omega}y}\\
	&=&\comp{G}{x}{y}+\comp{G}{\ultra{\omega}x}{\ultra{\omega}y}\\
	&>&\comp{G}{x}{y}
\end{eqnarray*}
-- a contradiction.
\end{proof}
This is perhaps an overly fancy way of stating the obvious, but there are interesting consequences to this approach. One is a clear linkage between compressibility and strong (!!) collapsibility:
\begin{prop} If $\bd X$ contains no non-degenerate incompressible set, then $\bd X$ is strongly collapsible.\ep
\end{prop}
\begin{proof} Assume that $\bd X$ contains no non-degenerate incompressible set. By the preceding lemma, $\bdCone X$ has a unique minimal non-empty closed $G$-invariant subset $M$. Moreover, $\Delta_M=\left\{(x,x)\,\big|\,x\in M\right\}$ is the only minimal non-empty closed $G$-invariant subset of $M\times M$ (the action of $G$ is the diagonal action). 

By lemma 2.4 of \cite{[Ballmann-Buyalo-2pi]}, $\Delta^d_M=\left\{(x,\ldots,x)\in (\bd X)^d\,\big|\,x\in M\right\}$ is then the only minimal non-empty closed $G$-invariant subset of $(\bd X)^d$ for all $d\geq 1$, implying that every finite subset $A\subset\bd X$ is strongly collapsible. 

By corollary \ref{pointwise convergence of G-omegas}, and by the fact that every $\ultra{\omega}$ is a continuous map of $\bdTits X$, this implies that any Tits-compact set $A\subset\bd X$ is strongly collapsible as well. 

Finally, theorem A produces an $\omega\in\UF$ for which $\ultra{\omega}(\bd X)$ is a round sphere $\SS$. Since $\SS$ is Tits-compact, there is a $\nu\in\UF$ with $\ultra{\nu}\SS$ a singleton. But then $\ultra{\nu\cdot\omega}(\bd X)$ is a singleton.
\end{proof}
\begin{lemma}\label{lemma:incompressibles contained in core} Suppose $A$ is a non-degenerate incompressible set and let $n\in\bdTits X$. Then $A\subseteq\Titsball{n}{\pi}$. In other words, every non-degenerate incompressible set is contained in $\core$.
\end{lemma}
\begin{proof} Since $A$ is non-degenerate and incompressible, $G$ must be of higher rank, and $\bdTits X$ is bounded.

Let $\omega\in\beta G$ be chosen to pull away from $n$, and let $\connect{\omega}{n}{p}$. By $\pi$-convergence, $\ultra{\omega}m=p$ whenever $\Tits{n}{m}\geq\pi$. Since $A$ is incompressible, this means at most one $a\in A$ satisfies $\Tits{n}{a}\geq\pi$.

Suppose $a\in A$ satisfies $\Tits{n}{a}>\pi$ and denote $q=\ultra{\omega}n$. We need to derive a contradiction.

The case $n\in A$ is impossible, because we already know $\diam(A)\leq\pi$. Since $A$ is non-degenerate and $a$ is the only point outside of $\Titsball{n}{\pi}$, there is a point $b\in A\cap\Titsball{n}{\pi}$.

Then, by corollary \ref{UFs map to suspensions}, $b$ is mapped to a point in $\suspend{q}{p}\minus\{p,q\}$. However, $\Tits{q}{\ultra{\omega}b}=\Tits{n}{b}$ by lemma \ref{embedding lemma}, while $\Tits{\ultra{\omega}b}{p}=\Tits{b}{a}$ by incompressibility. Thus $\Tits{n}{a}\leq\pi$ -- a contradiction. 
\end{proof}
Recall that $G$ has rank one iff $\core$ is empty (by proposition \ref{prop:Ballmann-Buyalo}). At the same time, $G$ having rank one implies $\II$ is empty by lemma \ref{rank one is strongly collapsible}. In view of the rank rigidity conjecture we would like to raise the following --
\begin{conjecture}\label{conj:strongly collapsible boundary implies rank one} $\bdCone X$ is strongly collapsible if and only if $G$ has rank one. 
\end{conjecture}
Affirming the conjecture will mean $\II$ is the only obstruction to $G$ having rank one.

Another aspect of lemma \ref{lemma:incompressibles contained in core} is that it proposes an approach to the Closing Lemma (conjecture \ref{conj:closing lemma}):
\begin{cor}\label{approach to closing lemma} Suppose $G$ is a group acting geometrically on a CAT(0) space $X$. If $\II$ covers $\bd X$, then $G$ has higher rank and $\diam\bdTits X=\pi$.
\end{cor}
In fact, it suffices to require a weaker condition to affirm the Closing Lemma:
\begin{cor}\label{easier approach to closing lemma} Suppose $G$ is a group acting geometrically on a CAT(0) space $X$ and let $\SS$ be the sphere produced by theorem A. If $\II$ covers $\SS$, then $G$ has higher rank and $\diam\bdTits X=\pi$.
\end{cor}
\begin{proof} We apply lemma \ref{quantified extremality} again. If some $p\in\bd X$ is an $\epsilon$-extreme point for some positive $\epsilon$, it suffices to find a point $n\in\SS\cap\hsm{D}(p)$ in order to conclude that the antipode of $n$ on $\SS$ must be extreme and $\II$ fails to cover $\SS$.

To find such a point, observe that $\hsm{D}(p)$ contains an element of $\minG$. Recalling that $\SS$ intersects every element of $\minG$ finishes the proof.
\end{proof}
Our conjecture is:
\begin{conjecture}\label{conj:higher rank boundary is covered by incompressibles} If $G$ has higher rank, then $\II$ covers $\bd X$.
\end{conjecture}

\subsection{The Geometry of Incompressible Sets}\label{subsec:geometry of incompressibles} We now turn to studying the geometry of individual elements of $\II$, with the aim to prove theorem C from the introduction.
\begin{lemma} If $A$ is an incompressible set, then its Tits closure $\clTits{A}$ is incompressible too.
\end{lemma}
\begin{proof} Let $\omega\in\beta G$ and let $B=\clTits{A}$. Consider $f(x,y)=\comp{\omega}{x}{y}$ as a function of $\bdTits X\times\bdTits X$ this time: the set $A$ is incompressible, so $f(A\times A)=0$; since $f$ is Tits-continuous, we conclude $f_\omega(B\times B)=0$. This being true for any $\omega$, we conclude $B$ is incompressible.
\end{proof}
\begin{cor} An incompressible subset of $\bdTits X$ has diameter less than or equal to $\pi$  and its closure in $\bdTits X$ is Tits-compact. In particular, a Tits-closed incompressible subset of $\bdTits X$ is Tits-compact (and hence cone-compact).
\end{cor}
\begin{proof} If $A$ is an incompressible subset of $\bdTits X$, set $B=\clTits{A}$. Use the folding lemma (theorem A) to find $\omega\in\beta G$ such that $\ultra{\omega}(\bd X)$ is a round sphere $\SS\subseteq\bdTits X$. Since $B$ is incompressible (last lemma), $\ultra{\omega}$ restricts to an isometric embedding of $B$ in $\SS$.  Since the Tits boundary is a complete metric space, the image of $B$ must be a closed subset of $\SS$, which implies $B$ is Tits-compact and has diameter less than or equal to $\pi$.
\end{proof}
The folding lemma has another related application:
\begin{lemma} Suppose $\{p,q\}\subset\bdTits X$ is an incompressible pair and $\gamma:[0,\Tits{p}{q}]\to X$ is a geodesic from $p$ to $q$. Then $\gamma\left([0,\Tits{p}{q}]\right)$ is an incompressible set.
\end{lemma}
\begin{proof} By the previous corollary, $\Tits{p}{q}\leq\pi$, and the rest follows from the fact that the composition of a $1$-Lipschitz map $f$ with a geodesic $\gamma$ is again a geodesic if and only if $f$ does not shrink the distance between the endpoints of $\gamma$. Here $f=\ultra{\omega}$ and $\gamma$ is a geodesic from $p$ to $q$. 
\end{proof}
\begin{prop}\label{max incompressible is convex} A maximal incompressible subset of $\bdTits X$ is connected and $\pi$-convex.
\end{prop}
\begin{proof} We prove $\pi$-convexity first. It suffices to prove that if $\{a,b,c\}$ is an incompressible triple with $\Tits{a}{b}<\pi$, then $\{c,p\}$ is incompressible for every $p\in[a,b]$ (recall that $\Tits{a}{b}<\pi$ implies there is only one geodesic arc joining $a$ with $b$).

Let $p\in[a,b]$ and $\omega\in\beta G$ be fixed arbitrarily and denote $a'=\ultra{\omega}a$, $b'=\ultra{\omega}b$, $c'=\ultra{\omega}c$ and $p'=\ultra{\omega}p$. By the preceding lemma, $\Tits{p'}{a'}=\Tits{p}{a}$, $\Tits{p'}{b'}=\Tits{p}{b}$. We also have $\Tits{p'}{c'}\leq\Tits{p}{c}$ since $\ultra{\omega}$ is 1-Lipschitz.

Next, the triple $\{a',b',c'\}$ is incompressible, so use the folding lemma to find $\nu\in\beta G$ and an isometrically embedded round $2$-sphere $\SS\subset\bdTits X$ of diameter $\pi$ such that $\ultra{\nu}\{a',b',c'\}\subset\SS$. Since the map $f=\ultra{\nu}\circ\ultra{\omega}$ restricts to an isometry on $\{a,b,c\}$ and since $\bdTits X$ is $\pi$-uniquely geodesic, $\SS$ contains $f([a,b])$ as well. Consequently, the spherical triangle $\bar\Delta$ with vertices $f(a),f(b),f(c)$ is a comparison triangle for $\vartriangle(a,b,c)$ and the point $f(p)$ is a comparison point for the point $p$. By the CAT(1) inequality, $\Tits{p}{c}\leq\Tits{f(p)}{f(c)}$. 
Thus by 1-Lipschitz
$$\Tits{p}{c} \le \Tits{f(p)}{f(c)}=\Tits{\ultra{\nu} p'}{\ultra{\nu} c'} \le \Tits{p'}{c'} \le  \Tits{p}{c} $$
Since $ \Tits{p'}{c'} =  \Tits{p}{c} $, the set $\{c,p\}$ is incompressible.

Let $A$ be a maximal incompressible set in $\bdTits X$. If $A$ is disconnected, then $A$ is the disjoint union of a pair of non-empty Tits-closed subsets $P$ and $Q$ (this is by Tits-compactness). By $\pi$-convexity of $A$, $\Tits{P}{Q}\geq\pi$. Since the diameter of $A$ is at most $\pi$, we have $\Tits{p}{q}=\pi$ for all $p\in P$ and $q\in Q$. Since $A$ is isometric to a subset of a standard unit sphere we must conclude both $P$ and $Q$ are degenerate. But then any geodesic joining them is an incompressible set containing $A$ as a proper subset -- a contradiction to maximality.
\end{proof}
This essentially finishes the proof of theorem C from the introduction:
\begin{cor}[`Theorem C'] Let $d$ be the geometric dimension of $\bdTits X$. Then every maximal incompressible subset $A\subseteq\bdTits X$ is isometric to a compact convex subset of $\SS^d$.
\end{cor}
\begin{proof} It suffices to consider $A\in\II$. Use the folding lemma to isometrically embed $A$ in a round sphere $\SS\subset\bdTits X$. A closed, connected, $\pi$-convex subset of a round sphere is the intersection of a family of closed hemispheres, and we are done.
\end{proof}
A more precise description of the elements of $\II$ is available in \ref{join decomposition of maximal incompressibles}, but we will need to work a bit more for it.

\subsection{The Sphere of Poles}
\begin{defn} A point $a\in\bdTits X$ is said to be a {\it pole}, if it has an antipode $b$ such that $\{a,b\}$ is incompressible, and we then say that $\{a,b\}$ is a {\it dipole}. We denote the set of all poles by $\fat{P}$.
\end{defn}
By incompressibility, it is clear that the action of $\beta G$ on $\bdTits X$ preserves dipoles and that $\Tits{a}{b}=\pi$ for any dipole $\{a,b\}$.

\begin{lemma} If $\{a,b\}$ is a dipole, then $\suspend{a}{b}=\bdTits X$.
\end{lemma}
\begin{proof} Suppose $n\in\bd X$ does not lie on a Tits geodesic joining $a$ to $b$. Find $\omega\in\UF$ pulling away from $n$. Write $\connect{\omega}{n}{p}$.

Since $\{a,b\}\subset\core$ by lemma \ref{lemma:incompressibles contained in core}, $\ultra{\omega}$ maps the geodesics $[n,a]$ and $[n,b]$ isometrically to longitudes of the suspension $\suspend{\ultra{\omega}n}{p}$. Since $\Tits{n}{a}+\Tits{n}{b}>\pi$, we conclude that $\Tits{\ultra{\omega}n}{\ultra{\omega}a}+\Tits{\ultra{\omega}n}{\ultra{\omega}b}<\pi$ -- a contradiction to $\{a,b\}$ being a dipole. 
\end{proof}
Using Swenson's theorem \ref{join decomposition}, we obtain:
\begin{thm}\label{join decomposition of maximal incompressibles} The set $\fat{P}$ of all poles coincides with the sphere $\SS(\bdTits X)$ of all suspension points of $\bdTits X$. Moreover, every maximal incompressible subset $A$ of $\bdTits X$ contains $\fat{P}$ and decomposes as a join $\fat{P}\ast E'(A)$ where $E'(A)$ is a connected, $\pi$-convex, maximal incompressible subset of $E'(\bdTits X)$, with radius strictly less than $\pi/2$.
\end{thm}
\begin{proof} Write $\bdTits X=\SS(\bdTits X)\ast E'(\bdTits X)$. By the uniqueness of Swenson's decomposition, both $\SS(\bdTits X)$ and $M=E'(\bdTits X)$ are $G$-invariant and hence $\beta G$-invariant. Moreover, the cone topology on $\SS(\bdTits X)$ coincides with the Tits topology, so $\beta G$ acts on $\SS(\bdTits X)$ by isometries. In particular, $\SS(\bdTits X)$ is incompressible, and therefore consists entirely of poles. We conclude that $\SS(\bdTits X)$ coincides with the set $\fat{P}$ of poles.

Let now $A$ be a maximal incompressible subset of $\bdTits X$. Since both $\fat{P}$ and $M$ are $\beta G$-invariant, join segments of the decomposition are mapped isometrically to join segments, implying that $A\cup\fat{P}$ is incompressible. By the maximality of $A$, this means $\fat{P}\subseteq A$.

Let $d$ be the geometric dimension of $\bdTits X$ and use the folding lemma to produce $\omega\in\beta G$ and a round sphere $\SS\subset\bdTits X$ of dimension $d$ such that $\ultra{\omega}\bd X=\SS$. Since $\fat{P}$ is incompressible and $\beta G$-stable we must have $\fat{P}\subseteq\SS$. Decompose $\SS$ as the spherical join of $\fat{P}$ with a round sub-sphere $\SS'\subseteq\SS$.

We also have $\fat{P}=\ultra{\omega}\fat{P}\subseteq\ultra{\omega}A$. Consider the join decomposition of $B=\ultra{\omega}A$: on one hand we have $\fat{P}\subseteq\SS(B)$, while on the other the incompressibility of $B$ implies that every pair of suspension points of $B$ is a dipole. Thus $\SS(B)=\fat{P}$, implying $\SS(A)=\fat{P}$. Turning to $E'(B)\subset\SS'$, we observe $\diam E'(B)<\pi$. For a closed, connected $\pi$-convex subset of a round sphere this implies it is contained in an open hemisphere and we must conclude that $E'(B)$ -- and hence also $E'(A)$ -- has radius less than $\pi/2$. 
\end{proof}
From this we obtain the following easy corollary:
\begin{cor}[Closing Lemma for the Poor] Suppose $G$ is a group acting geometrically on a CAT(0) space $X$. If $\bdTits X$ contains a pole, then $\diam\bdTits X=\pi$.
\end{cor}

\subsection{The Main Theorem} We are practically ready to prove theorem B.

\begin{lemma} Suppose a group $G$ acts geometrically on a CAT(0) space $X$. If $G$ stabilizes some $A\in\II$, Then $E'(A)$ is empty.
\end{lemma}
\begin{remark} We only require $G$ to stabilize $A$ as a set, that is: $G\cdot A=A$. 
\end{remark}
\begin{proof} Suppose this lemma is false. Then there exists a geometric action of a group $G$ on a CAT(0) space $X$ such that $G$ stabilizes a maximal incompressible set $A$ with $E'(A)$ non-empty and $\dim\fat{P}$ is smallest possible.

Since $E'(A)$ is closed a connected and $\pi$-convex $G$-invariant set of radius strictly less than $\pi/2$, $E'(A)$ has a center $z\in E'(A)$, and $G$ fixes $z$.

By a result of Ruane (see \cite{[Ruane-dynamics]}), $G$ virtually splits as a product $H\times\langle g\rangle$, with $g$ hyperbolic and $\bd X=\{g^{\pm\infty}\}\ast \bd Z$ where $Z$ is a closed convex $H$-co-compact subspace of $X$. Observe that $\{g^{\pm\infty}\}\subseteq\fat{P}$, implying that $\bd Z$ decomposes as $E'(\bdTits X)\ast\fat{P}'$ where $\fat{P}'$ is the equator of the pair $\{g^{\pm\infty}\}$ in $\fat{P}$, and this decomposition is $H$-invariant, as $H$ fixes both of the points $g^{\pm\infty}$. 

It follows that $\dim\SS(\bd Z)<\dim\fat{P}$, and $H$ acts geometrically on $Z$ while stabilizing the maximal ($H$-)incompressible set $A\cap\bd Z$. The minimality property of the triple $(G,X,A)$ then implies $A\cap\bd Z=\SS(\bd Z)\subset\fat{P}$, which is absurd: $A\cap\bd Z$ contains $E'(A)$ which is disjoint from $\fat{P}$.
\end{proof}
We are now ready to prove theorem B. Here is its more technical formulation:
\begin{thm} Suppose $G$ is a group acting geometrically on a CAT(0) space $X$ and let $d$ denote the geometric dimension of $\bdTits X$. If $G$ stabilizes an element $A\in\II$, then $G$ contains a free-Abelian subgroup $H$ of finite index and $X$ contains an $H$-periodic coarsely-dense $(d+1)$-flat $F\subseteq X$. In particular: $A=\bd X$ and $\bdTits X$ is a round sphere.
\end{thm}
\begin{proof} Let $A$ be a maximal incompressible subset of $\bd X$ which is stabilized by $G$. The preceding lemma shows $E'(A)$ is empty, and theorem \ref{join decomposition of maximal incompressibles} implies $A=\fat{P}$. Moreover, $A$ is the only non-degenerate maximal incompressible subset of $\bd X$.

Since $A$ is cone-compact (because it is Tits-compact), it follows that $\beta G$ stabilizes $A$. For any point $q\notin A$, we have that $A\cup\{q\}$ is a compressible set. Consider the function $f:\UF\to\RR$ defined by $f(\omega)=\Tits{A}{\ultra{\omega}q}$. By lower semi-continuity, $f$ attains its minimum at some $\omega_0\in\UF$. The invariance of $A$, however, implies this minimum must equal zero (or else $\ultra{\omega}(A\cup\{q\})=A\cup\{\ultra{\omega_0}q\}$ is an incompressible set strictly containing $A$). 

Thus, every point of $\bd X$ can be collapsed onto $\fat{P}$, implying $E'(\bdTits X)$ is empty, so that $\bd X=\fat{P}$. 

In particular, $\fat{P}$ bounds a $(d+1)$-flat $F'$ in $X$. By lemma 10 in \cite{[Swenson09]}, $F'$ is coarsely-dense in $X$, so that $G$ is quasi-isometric to a free-abelian group. 

By a well-known result of Shalom \cite{[Shalom-virtually-abelian]}, $G$ must be virtually Abelian. The flat torus theorem then provides the flat $F$ we are looking for.
\end{proof}

\section{Discussion}\label{section:discussion}
Henceforth let $d$ denote the geometric dimension of $\bdTits X$, and set \[\displaystyle d'=\max_{A\in\II}\dim(A)\,.\] 

Using the folding lemma we fix $\omega_0\in\beta G$, a maximal flat $F_0\subset X$ with boundary sphere $\SS_0$ so that $\ultra{\omega_0}$ maps $\bd X$ onto $\SS_0$.

Every $A\in\II$ may then be assigned a {\it volume} $\vol{A}$ equal to the $d'$-dimensional volume of $\ultra{\omega_0}A$ measured in the round sphere $\SS_0$. By incompressibility, this notion of volume is independent of the choice of $\omega_0$ and $\SS_0$.

Let $\hsm{P}$ denote the space of all compact convex (that is, compact connected $\pi$-convex) $d'$-dimensional polytopes contained in the sphere $\SS_0$. When endowed with the Hausdorff metric, $\hsm{P}$ is a compact space with the $d'$-dimensional volume function a continuous non-negative real-valued function on this space. We conclude that $\II$ has elements of maximal volume -- denote the collection of these elements by $\IImax$.
\begin{lemma}\label{max volume implies max non-collapsible} Suppose $A\in\IImax$. Then $A$ is a maximal non-collapsible set. 
\end{lemma}
\begin{proof} Since $A$ is incompressible, it is not collapsible. Take any point $z\in\bd X$, and consider the function $f(\omega)=\Tits{A}{\ultra{\omega}z}$ on $\beta G$. Let $\nu$ be a minimum point of this function. If $f(\nu)>0$, then $\ultra{\nu}A\cup\{\ultra{\nu}z\}$ is incompressible, which is impossible, by the maximality properties of $d'$ and $\vol{A}$. We conclude that $f(\nu)=0$ and $\ultra{\nu}z\in\ultra{\nu}A$, proving that $A$ is a maximal non-collapsible set.
\end{proof}
\begin{prop}[Strong Folding Lemma (Theorem D)]\label{strong folding} Let $A\in\IImax$. Then:
\begin{enumerate}
	\item $\ultra{\omega}A\in\IImax$ for all $\omega\in\beta G$;
	\item There exists $\nu\in\beta G$ such that $\ultra{\nu}\bd X=\ultra{\nu}A$.
\end{enumerate}
\end{prop}
\begin{proof} The first claim follows from the fact that $\ultra{\omega}A$ is incompressible: $\ultra{\omega}A$ is contained in an element $\tilde A$ of $\II$, but, being of maximal volume, $\ultra{\omega}A$ must coincide with $\tilde A$.

We claim that every compact subspace $K$ of $\bdTits X$ has $\nu'\in\beta G$ such that $\ultra{\nu'}K\subset\ultra{\nu'}A$. By the Tits-continuity of all $\ultra{\omega}$, $\omega\in\beta G$, it suffices to fix a countable dense subset $D$ of $K$ and find $\nu'\in\beta G$ with $\ultra{\nu'}x\in\ultra{\nu'}A$ for all $x\in D$. Writing $D=\{d_n\}_{n=1}^\infty$, use the preceding lemma to realize the following inductive construction:
\begin{itemize}
	\item $\nu'_1\in\beta G$ with $\ultra{\nu'_1}d_1\in\ultra{\nu'_1}A$.
	\item Suppose $\nu'_k\in\beta G$ with $\ultra{\nu'_k}\{d_1,\ldots,d_k\}\subset\ultra{\nu'_k}A$. By (1) and the preceding lemma, $\ultra{\nu'_k}A$ is a maximal non-collapsible set, so there exists an $\alpha_k\in\beta G$ such that $\ultra{\alpha_k}\left(\ultra{\nu'_k}d_{k+1}\right)\in\ultra{\alpha_k}\left(\ultra{\nu'_k}A\right)$. Thus, $\nu'_{k+1}=\alpha_k\cdot\nu'_k$ satisfies $\ultra{\nu'_{k+1}}\{d_1,\ldots,d_{k+1}\}\subset\ultra{\nu'_{k+1}}A$. 
\end{itemize}
We conclude that none of the sets
\begin{displaymath}
	\Omega_k=\left\{\omega\in\beta G\,\Big|\,\ultra{\omega}d_n\in\ultra{\omega}A\,,\;\;n=1,\ldots,k\right\}
\end{displaymath}
are empty. By continuity of the map $\omega\mapsto\ultra{\omega}$, the sets above are all closed subsets of $\beta G$. We conclude there is an element $\nu'\in\bigcap_{k\in\NN}\Omega_k$. This is the $\nu'\in\beta G$ we are looking for.\\ 

To finish the proof of the proposition, we pick $K=\SS_0$ so that $\ultra{\nu'\cdot\omega_0}\bd X=\ultra{\nu'}\SS_0\subset\ultra{\nu}A$. Thus, the second claim holds with $\nu=\nu'\cdot\omega_0$, as desired.
\end{proof}
An immediate conclusion from this result is:
\begin{cor}\label{all IImax are isometric} Fix $A_0\in\IImax$. Then all elements of $\IImax$ are isometric to $A_0$ and every incompressible subset of $\bdTits X$ is isometric to a subset of $A_0$.
\end{cor}
\begin{proof}
Fix $A,B\in\IImax$ and let $\nu\in\beta G$ satisfy $\ultra{\nu}\bd X=\ultra{\nu}A$. Then $\ultra{\nu}B\subseteq\ultra{\nu}A$ and the maximality of $\vol{B}$ guarantees $\ultra{\nu}B=\ultra{\nu}A$. By incompressibility of $A$ and $B$, $A$ is isometric to $B$.
\end{proof}
\begin{cor}\label{special IImax element} There are a flat $F\subset X$ and an element of $A_0\in\IImax$ such that $A_0\subseteq\bd F$ and $A_0$ intersects every element of $\minG$.
\end{cor}
\begin{proof} Once again, let $A\in\IImax$ and $\nu\in\beta G$ satisfy $\ultra{\nu}\bd X=A$. Let $F_0$ and $\omega_0\in\beta G$ be as above. Set $A_0=\ultra{\omega_0}A$. Then every $M\in\minG$ has $\ultra{\omega_0\cdot\nu}M\subseteq M\cap A_0$. 
\end{proof}
At present, little is known about how distinct elements of $\IImax$ interact. 
%
%
%
%
We conjecture that spheres such as $\SS_0$ are {\it tiled} by isometric copies of $A_0$, as is the case for co-compact lattices of Euclidean buildings. This hints at the possibility of $\bigcup\IImax$ carrying the structure of a building whose apartments are such tiled spheres. This makes us hopeful regarding the role our approach could play in resolving the Rank Rigidity Conjecture: the union of $\IImax$ forms a $\beta G$-invariant structure in $\bdTits X$; though recovering the properties of a building for this structure directly may seem like a hard task at the moment, we believe it should be possible to use $\IImax$ in conjunction with existing metric characterizations of spherical buildings to either prove or disprove the conjecture.

\bigskip

\bigskip

\bibliographystyle{siam}
\bibliography{rigidityref}

\end{document}